\documentclass[sts]{imsart}

\RequirePackage{amsthm,amsmath,amsfonts,amssymb}
\RequirePackage[authoryear]{natbib}

\startlocaldefs

\usepackage{algorithm2e}
\usepackage[utf8]{inputenc}

\usepackage{hyperref} 
\hypersetup{	colorlinks=true,	linkcolor=blue,	filecolor=magenta,	urlcolor=blue,	allcolors=blue,}

\usepackage{yun}
\newtheorem{definition}[thm]{Definition}

\usepackage{graphicx}
\usepackage{graphics}
\graphicspath{{graphs/}} 
\usepackage{pinlabel}

\def \P{P}
\def \Pb{P} 
\def \Qb{Q} 

\def \myf{\phi}
\def \myF{\Phi}

\def \ker{\operatorname{ker}}

\def \KS{D_{\operatorname{KS}}}
\def \MMD{D_{\operatorname{MMD}}}

\renewcommand{\myred}[1]{{#1}}
\newcommand{\LN}[1]{ {\color{red} LN: {#1}  }}
\renewcommand{\myredvsec}[1]{{#1}}

\def \n{{m}}
\def \m{{N}}
\def \Acalcal{\mathcal{X}}
\def\Xfrak{\mathfrak{X}}
\def \R{\mathbb{R}}
\def \thetave{\theta}
\def \myD{\ensuremath{D}}
\def \Ecal{\mathcal{E}}
\def \thetaeta{\theta}

\def \Dscr{\mathscr{D}}
\def \Dset{\mathbb{D}}
\def \Dcal{\mathcal{D}}
\def \Pcal{\mathcal{P}}
\def \Kcal{\mathcal{K}}
\def \Hcal{\mathcal{H}}
\def \Q{Q}
\def\1ve{\bm{1}}

\newcommand{\W}{\ensuremath{W}}
\newcommand{\Wr}{\ensuremath{W_r}}

\newcommand{\Wcalr}{\ensuremath{\W}_r}

\newcommand{\TPlan}{\ensuremath{\mathcal{T}}}

\def \dd{\textrm{d}}

\endlocaldefs

\begin{document}

\begin{frontmatter}
\title{Optimal transport based theory for latent structured models}
\runtitle{Optimal transport based theory for latent structured models}

\begin{aug}
\author[A]{\fnms{XuanLong}~\snm{Nguyen}\ead[label=e1]{xuanlong@umich.edu}}
\and
\author[B]{\fnms{Yun}~\snm{Wei}\ead[label=e2]{yun.wei@utdallas.edu}}
\address[A]{University of Michigan\printead[presep={\ }]{e1}.}
\address[B]{University of Texas at Dallas\printead[presep={\ }]{e2}.}
\end{aug}

\maketitle

\begin{abstract}
This article is an exposition on some recent theoretical advances in learning latent structured models, with a primary focus on the fundamental roles that optimal transport distances play in the statistical theory. We aim at what may be the most critical and novel ingredient in this theory: the motivation, formulation, derivation and ramification of inverse bounds, a rich collection of structural inequalities for latent structured models which connect the space of distributions of unobserved structures of interest to the space of distributions for observed data. This theory is illustrated on classical mixture models, as well as the more modern hierarchical models that have been developed in Bayesian statistics, machine learning and related fields. 
\end{abstract}


\end{frontmatter}

\begin{comment}
\textbf{to do list}
\begin{enumerate}
    \item check whether notation is defined or not
    \item update everything about the words "proof" and "supplementary material"
    \item add proper references to the important theorems to our own papers
    \item fix the question mark at the end
    \item should we change the sentence something like "we do something" since it is a review paper and should use third person tone?
    \item 1/6: add something on improving the rate from $n^{-\frac{1}{4}}$ in overfitted case to $n^{-\frac{1}{2}}$??
\end{enumerate}
\end{comment}

\begin{comment}
\begin{enumerate}
\item 
Introduction: Page 1-3. LN
\item 
Minimum-distance: Page 4-9. YW
\item 
Strong and Weak identifiability: Page 10-15. YW
\item 
de Finetti's mixing measure: Page 16-21. LN and YW
\item 
Hierarchical models: 22-24. LN

\end{enumerate}
\end{comment}

\section{Introduction}
\label{sec:intro}

Estimating probability measures from observed data samples is one of the fundamental problems in statistics --- a probability measure of interest provides the law for a sample of the data population. The traditional approach in classical statistics is to assume that the probability measure admits a density function, parametrically or nonparametrically specified, with respect to a suitable dominating measure (e.g., Lebesgue measure on a Euclidean space). This leads to book length treatments for the topic of density estimation, e.g., \cite{Lehmann-Casella,vandeGeer-00,ghosal2017fundamentals}.
An equally fundamental, and arguably less developed, formulation for estimating probability measures is the setting where such measures describe patterns of some \emph{indirect} quantity of interest apart from the directly observed data population. 

Take, for instance, the problem of learning clustering patterns from a given data sample $\Dset=\{X_1,\ldots, X_n\}$ taking values in a space $\Xf$. Although one may simply apply a clustering algorithm directly on the data set (e.g., using something such as the $k$-means algorithm \cite{hartigan1985statistical}), to give inferential meaning to a data-driven procedure, one needs to endow a statistical model which provides the mathematical linkage from the observed data set $\Dset$ to the quantity of interest that represents the clustering structure. The most basic model assumes that $\Dset$ is an i.i.d. sample from a mixture of distributions, which admits a density function of the form
$X_1,\ldots, X_n \stackrel{i.i.d.}{\sim} \sum_{k=1}^{K} p_k f(\cdot|\theta_k)$. Here, $K \in \Nb \cup \{\infty\}$ represents the number of clusters, 
$p:=(p_1,\ldots, p_K) \in \Delta^{K-1}$ is a vector of mixing probabilities, and $f(x|\theta)$ is a pre-specified probability kernel on $\Xf \times\Theta$. The parameter $\theta_k \in \Theta$ represents a specific feature associated with the $k$-th mixture component, for $k=1,\ldots, K$. Thus, all that is unknown about the clustering pattern is represented by the (discrete) probability measure $G=\sum_{k=1}^{K}p_k \delta_{\theta_k}$. One can infer about the nature of data heterogeneity and the clustering composition through the number of atoms that provide the support for $G$, as well that the estimates of the atoms $\theta_k$ and associated mass $p_k$. The \emph{latent} $G$, also known as
mixing measure $G$, is the quantity of interest, even though one does not have direct samples of it; only sample(s) from the induced probability density 
\begin{equation}
\label{Eq:mixture}
p_G(x)  = \int f(x|\theta) G(\dd \theta) =\sum_{k} p_k f(x|\theta_k).
\end{equation}

Mixture models are the simplest instance in a range of probabilistic latent variable models and Bayesian hierarchical models, which shall be referred to as \emph{latent structured models} in this article. The range of such probabilistic modeling creatures is vast, but they share in common the presence of one or more latent mixing measures which represent interpretable quantities of inferential interest. Such latent structures increase the expressiveness of the model for the observed data, but they also carry domain-specific interpretation for the human practitioners. Latent structured models have played indispensable roles in many fields, from the analysis of stars and galaxies~\citep{mukherjee1998three} to that of genomes and phylogenies~\citep{pritchard2000inference,Felsenstein-81}, from the development of autonomous robots~\citep{thrun2005probabilistic} to information extraction engines over texts \citep{blei2003latent,tang2014understanding}, images~\citep{Sudderth-et-al08,ho2020denoisingdiffusionprobabilisticmodels} and network data \citep{HOLLAND1983109,amini2024hierarchical}. 
Because of the non-black box nature of such models, it is of interest to develop statistical theory about basic issues involving the identifiability and inference of latent mixing measures that arise in such latent structured models.

In this article we give an exposition on some recent developments on theoretical questions associated with the estimation of latent mixing measures arising in mixture models and Bayesian hierarchical models. 
Due to the versatility of mixture models, basic questions such as identifiability received early attention \citep{teicher1963identifiability,teicher1967identifiability}. Subsequent theoretical treatments concerning the estimation of mixing measures and related quantities have been studied by various authors (see, e.g., \cite{lindsay1995mixture,chen1995optimal,Dacunha-Castelle1997,ghosal2001entropies} and \cite{ho2020robust} for additional references). By contrast, comparable results for hierarchical models are rare, despite their proven effectiveness in applications across a vast range of data domains, as highlighted in the previous paragraph. The lack of progress may be due to both conceptual and technical challenges  --- on the one hand, a hierarchical model may be viewed itself as a mixture (of mixtures) model that anticipates the inference of multiple data populations; on the other hand, the presence of multiple layers of latent variables makes the resulting model quite complex, and they tend to stymie most straightforward attempts toward obtaining a point estimate of latent structures from data, from which one may be able to conduct an analysis of the asymptotic kind \citep{nguyen2015posterior,nguyen2016borrowing}.

There have been considerably substantial advances in the last 15 years in the statistical theory for latent structures, especially in the realm of mixture modeling.
In fact, for \emph{finite} mixtures of parametric kernels equipped with sufficiently strong identifiability conditions, optimal rates of estimation have been established in a pointwise \cite{nguyen2013convergence,ho2016strong,heinrich2018strong,wei2022convergence} or a minimax sense \citep{heinrich2018strong,wu2020optimal,wei2023minimum}. For general finite and infinite mixtures, new results were obtained for the first time, even if the full story remains incomplete \citep{nguyen2013convergence,ho2016convergence,ho2019singularity}. Promising progresses in the realm of hierarchical models have also been steadily made. Remarkably, most of these advances been been possible due partly to the recognition of the role that optimal transport plays in quantifying a suitable metric space of the mixing measures. 

The goal of this article is to describe such latest advances, with a focus on the natural and fundamental role of optimal transport in the statistical theories. Due to space constraint, we do not aim to be comprehensive, by covering only a selected number of latent structured models \myredvsec{(e.g., models for graphs and networks will be excluded; nor do we include important models such as hidden Markov models or other latent class models).} We hope to be technically incisive while keeping this exposition accessible, by zooming in and shining the light on several key concepts and essential technical ingredients that allow the aforementioned advances possible. At the heart of these theoretical results lie the various notions of identifiability and the entailment of a collection of structural results known as \emph{inverse bounds}. These are comparison inequalities that provide upper bounds on the optimal transport distances on the quantities of interest $G$, in terms of a suitable distance on the corresponding data distributions $p_G$. 
These types of inverse bounds enable one to transfer known characterization of the quality of any data density estimate into that of the mixing measure of interest without concerning any particular method of parameter estimation.

\subsection{Optimal transport distance and mixture models}
As a starter, we focus on the mixture model setting introduced above. In order to characterize the convergence behavior for an estimate of the mixing measure $G$, one needs a suitable metric. For finite and infinite mixture models, $G$ is typically a discrete measure. Thus, standard notions of distance metrics or divergences often employed in the statistical theory of density estimation, such as total variation distance, Hellinger distance or the Kullback-Leibler (KL) distance are no longer suitable. Optimal transport based distances arise as a natural candidate. Suppose that $\Theta$ is a subset of an Euclidean space equipped with the norm $\|\cdot\|$. Let $G,G' \in \Pcal(\Theta)$, the space of Borel probability measures on $\Theta$. Recall, e.g., from \cite{Villani-08},
for $r\geq 1$, the $r$-order Wasserstein distance is given by 
\begin{equation}
\label{eq:optdist}
W_r(G,G'):= \inf_{\kappa \in \TPlan(G,G')} \biggr [\int \|\theta-\theta'\|^r \textrm{d} \kappa(\theta,\theta') \biggr ]^{1/r},
\end{equation}
where $\TPlan(G,G')$ denotes the space of couplings of $G$ and $G'$, i.e., joint distributions on $\Theta\times \Theta$ that induce marginal distributions $G$ and $G'$. In this article, most of the theoretical results concerning the estimation of mixing measure $G$ will be characterized by a suitable metric $W_r$.  
\myredvsec{As noted in \cite{nguyen2013convergence}, convergence of a sequence of discrete probability measures in a Wasserstein distance metric toward a discrete mixing measure $G$ implies the convergence of subsets of atoms that provide support for the converging measures toward that of $G$, thereby providing a natural interpretation of convergence of clusters in the clustering application described above.}

Although the adoption of a distance metric arising from the theory of optimal transport might initially appear to be merely out of technical convenience, it belies a fundamental shift from the classical problem of estimating the data population's distribution and that of the unobserved latent structure (such as the mixing measure $G$). It is important to recognize that the space $\Xf$ where the data reside and the space that we use to represent the quantity of interest are typically not the same. It can be said that the modern asymptotic theory of density estimation and hypothesis testing is the theory of divergence measures --- such divergence measures arise as the "right" notion of distance to characterize the quality of statistical inference from data \citep{LeCam-86,Cover91,van2000asymptotic,ghosal2017fundamentals}.  Many are instances of a broader class of divergences known as $f$-divergence \citep{Ali-Silvey-66,Csiszar-67}: let $\phi:[0,\infty] \rightarrow [0,\infty]$ be a proper convex function, then an $f$-divergence between two probability density functions $f_1$ and $f_2$ with respect to a dominating measure $\lambda$ on  $\Xf$ is given by $\rho_\phi(f_1,f_2)= \int \phi(f_2/f_1) f_1 \textrm{d} \lambda$. Take $\phi(u)= (1/2)(\sqrt{u}-1)^2$; $\phi(u) = (1/2)|u-1|$ and
$\phi(u) = -\log(u)$ gives rise to the squared Hellinger distance (denoted $h^2$), the total variation distance ($V$) and the KL distance ($K$), respectively.

Since the space $\Theta$ provides a parameterization for the probability kernel $f(x|\theta)$ which composes the mixture model $p_G(x)$ via \eqref{Eq:mixture}, an optimal transport distance that inherits the metric structure of the mixture's kernels is
\begin{equation}
\label{eqn:composite}
d_{\rho_\phi}(G,G'):= \inf_{\pi \in \Kcal(G,G')} \int \rho_\phi(f(\cdot|\theta), f(\cdot|\theta')) \textrm{d} \pi(\theta,\theta'),
\end{equation}
where we have merely utilized the $f$-divergence $\rho_\phi$ as the unit cost of mass transportation for the parameter space $\Theta$. This is called a \emph{composite transportation distance}  \citep{nguyen2013convergence,nguyen2011wasserstein}. Moreover, one may "connect" the space of mixing measure $G$ to the corresponding space of data population's density function by the following \citep{nguyen2013convergence}:
\begin{lem} 
\label{Lem:struct} Provided that $d_\phi(G,G')$ and $\rho_\phi(p_G,p_{G'})$ are finite, then
$\rho_\phi(p_G,p_{G'}) \leq d_{\rho_\phi}(G,G')$.
\end{lem}
The relationship between $f$-divergence $\rho_\phi(p_G,p_{G'})$ and the corresponding composite transportation distance between $G,G'$ may be simplified further given specific choices of the kernel $f$ and the function $\phi$. For example, if $\Theta = \Rb^d$, $f(x|\theta)$ is the multivariate normal density kernel $N(\theta,I_{d\times d})$, and $\phi(u)={1/2}(\sqrt{u}-1)^2$, then $\rho_\phi$ becomes squared Hellinger distance: $h^2(f(\cdot |\theta), f(\cdot|\theta')) = (1/8)\|\theta-\theta'\|^2$. As a result, one obtains for Gaussian location mixture models: $h^2(p_G,p_{G'}) \leq (1/8)W_2^2(G,G')$. Since $V\leq h$, we also have $V(p_G,p_{G'}) \leq (1/8)W_2^2(G,G')$. Take $\phi(u)=-\log(u)$ instead, one obtains an inequality for the KL distance:
$K(p_G,p_{G'}) \leq (1/2) W_2^2(G,G')$.

\subsection{Inverse bounds and identifiability conditions}
\label{sec:inverseintro}
Inverse bounds represent a kind of reverse of the structural inequalities given in Lemma \ref{Lem:struct}. An example of such inequalities is of the following form: for some $r\geq 1$, and some \myredvsec{continuous and strictly increasing function $\Psi:[0,\infty)\rightarrow [0,\infty)$ such that $\Psi(0)=0$},
\begin{equation}
\label{Eq:inverse}
W_r(G,G') \leq \Psi(V(p_G,p_{G'})).
\end{equation}
This inequality, if true, provides finer information about the inverse of the integral mapping $G\mapsto p_G$. First, it implies that if $V(p_G,p_{G'}) = 0$, then $W_r(G,G')=0$. In other words, the mapping is injective, i.e., the unobserved mixing measure $G$ is uniquely identifiable from the data population's distribution $p_G$. Second, any characterization of the convergence behavior in the space of mixture densities $p_G$ may be transferred to that of the mixing measure $G$. This allows us to inherit from the tools and results in the existing literature concerning the convergence of data density estimation procedures to derive rates of convergence for the latent structures of interest. In fact, for both purposes, it is sufficient to establish ~\eqref{Eq:inverse} when $W_r(G,G')$ is sufficiently small. 
An equivalent form of such a claim is, where $\Psi^{-1}$ denotes the inverse of $\Psi$:
\begin{equation}
\label{Eq:inverse-local}
\liminf_{W_r(G,G') \rightarrow 0} \frac{V(p_G,p_{G'})}{\Psi^{-1}(W_r(G,G'))} >0.
\end{equation}

Unlike the generality of the inequalities given in Lemma \ref{Lem:struct}, inverse bounds of the form \eqref{Eq:inverse} may only hold subject to additional assumptions about the kernel $f$. They are also typically more challenging to obtain. Assumptions on the kernel $f$ are part of what is known as \emph{identifiability conditions}. Strictly speaking, identifiability conditions are originally imposed so that the latent $G$ is uniquely identified from $p_G$. In this article, as is the customary in the recent literature, we regard identifiability conditions as a broad class of conditions concerning either $f$ or $G$ which entail the inverse bounds such as \eqref{Eq:inverse} or \eqref{Eq:inverse-local}. 

One of the more general inverse bounds was first obtained for convolution mixture models.
Specifically, suppose that $f$ is a translation-invariant probability kernel that is 
symmetric around 0 in $\Rb^d$, i.e., $\Xf = \Rb^d$ and $\Theta\subset \Rb^d$, 
$f(x|\theta) := f(x-\theta)$ such that $\int_{B} f(x) dx = \int_{-B} f(x)dx$ for any Borel set $B \subset \Rb^d$. For convolution mixtures, the procedure for recovering the mixing measure $G$ from the data density is known as deconvolution, although the early literature on deconvolution was typically restricted to assuming that $G$ also has a density function on $\Rb^d$ \citep{Fan-91,Zhang-90,gassiat2022deconvolution}. We do not require such a restriction here; more frequently for us $G$ is discrete.
In general, the relevant identifiability conditions for convolution mixtures concern the smoothness of the kernel function $f$, which can be specified in terms of the tail behavior of its Fourier transform.
Assume that the Fourier transform of $f$ satisfies $\tilde{f}(\omega) \neq 0$ for all $\omega \in \Rb^d$.
We say $f$ is \emph{ordinary smooth} with parameter $\beta >0$ if
$\int_{[-1/\delta,1/\delta]^d} \tilde{f}(\omega)^{-2} d\omega
\lesssim (1/\delta)^{2d\beta}$ as $\delta \rightarrow 0$.
Say $f$ is \emph{supersmooth} with parameter $\beta > 0$
if $\int_{[-1/\delta,1/\delta]^d} \tilde{f}(\omega)^{-2} d\omega
\lesssim \exp (2d\delta^{-\beta})$ as $\delta \rightarrow 0$.
The following inverse bound was obtained \cite{nguyen2013convergence}.
\begin{thm} 
\label{thm:convolution}
For a suitable function $\Psi$, the inverse bound \eqref{Eq:inverse}
holds for $r=2$ for any pair $G,G'\in \Pcal(\Theta)$, provided that
$\Theta$ is a bounded subset of $\Rb^d$, and $W_2(G,G')$
is sufficiently small. In particular,
if $f$ is ordinary smooth with parameter $\beta$ (e.g. $f$ is a Laplace kernel), then
$\Psi(u) = u^{1/(2+\beta d')}$ for any $d' > d$.
If $f$ is supersmooth then $\Psi(u) = (-\log u)^{-1/\beta}$ (e.g., for normal kernel, $\beta=2$).
\end{thm}
Additional inverse bounds were subsequently obtained for $W_r, r\geq 1$, for Laplace kernels, using a similar technique \citep{gao2016posterior}, and extended to a more general class of optimal transport distances known as Orlicz-Wasserstein distance, for the normal kernels \citep{guha2023excess}. 
See also \cite{rousseau2024wasserstein} for related results and further improvements.

From here, one obtains almost immediately the rates of convergence for a variety of estimation procedures of the mixing measure $G$, given an $n$-iid sample $\Dset = \{X_1,\ldots, X_n\}$, by drawing on a body of known results in density estimation. Indeed, assume that $\Dset$ is an i.i.d. sample from $p_{G_0}$, a mixture of normal kernels, for some "true" mixing measure $G^*=G_0 \in \Pcal(\Theta)$ (note here that the main assumption on $G_0$ is that its support is bounded, but $G_0$ may be either discrete or a non-atomic measure).  Consider either a Bayesian estimation or a maximum likelihood estimation procedure, such as ones studied in \cite{ghosal2001entropies}, it can be shown that the posterior distribution or the MLE contracts to the true data density $p_{G_0}$ at a rate $(\log n)^{\kappa}/n^{1/2}$ for some constant $\kappa$ under the Hellinger distance. In particular, let $\hat{G}_n$ denote a point estimate obtained by either procedure for the mixing measure $G$, then $h(p_{\hat{G}_n}, p_{G_0}) = O_p(\log n)^{\kappa}/n^{1/2}$. Now, applying Theorem \ref{thm:convolution} using $\Psi(u)=(-\log u)^{-1/2}$ for the normal kernel yields $W_2(\hat{G}_n,G_0) = O_p((\log n)^{-1/2})$, which is known as the optimal minimax rate for the deconvolution problem with normal kernel \citep{dedecker2013minimaxratesconvergencewasserstein,Fan-91}. For Laplace mixtures, likewise, the inverse bound under $W_1$ from \cite{gao2016posterior} may be invoked to obtain the minimax optimal rate of deconvolution for a suitable MLE or Bayesian procedure.

We hope the above very brief sketch of theoretical results illustrates clearly how the inverse bounds effectively encapsulate the structural properties of a latent structured model $p_G$ without concerning any specific estimation procedure in mind, and one can easily obtain a characterization of the convergence behavior for $G$ as soon as the convergence analysis of the data population $p_G$ is achieved. Thus, it is our view that the derivation of inverse bounds plays a fundamental role in the theoretical investigation of latent structured models. However, as we will see later in this article, the inverse bounds can also be a powerful methodological tool that leads to concrete estimation methods for the latent structure $G$.

\subsection{Paper organization}

For infinite mixture models and hierarchical models, other than the convolution models mentioned above, there has been very few results regarding identifiability and convergence behavior of the latent structure (an exception is the work \cite{nguyen2016borrowing} to the best of our knowledge). On the other hand, most substantial progress in the past decade has been achieved for the setting of \emph{finite} mixture and hierarchical models. This will occupy a relatively large portion in the remainder of this article. There is now a rich body of work on finite mixture models where one may obtain optimal rates of estimation, subject to sufficiently strong identifiability conditions. When such strong identifiability conditions are violated, we shall enter the more mysterious realm of singular finite mixture models, of which a complete story remains to be told. This will be the topics of Section 2 and Section 3. 
In Section 4, we turn to more complex forms of mixture and hierarchical models. The overarching theme is the development of inverse bounds for the latent structured models in question. We will discuss the conditions required for such bounds to hold, the consequences of such structural results on known estimation procedures, and where applicable, how they motivate new procedures with favorable theoretical guarantees.

\subsection{Notation}
We collect a few common notational choices to be used throughout the article. 
%
$\Xf$ denotes the space where observed data points reside, while $\Theta$ stands for the parameter space from which we define a quantity of interest, such as $G\in \Pcal(\Theta)$, where $\Pcal(\Theta)$ stands for the space of Borel probability measures on $\Theta$. Here, $\Theta$ is typically a subset of a Euclidean space such as $\Rb^q$. $\Gc_k(\Theta)$ denotes the set of all discrete probability measures with at most $k$ atoms in $\Theta$, while $\Ec_k(\Theta)$ contains only probability measures that has exactly $k$ atoms. On the other hand, $\Xf$ is equipped with the sigma algebra $\Xc$; depending on examples sometimes $\Xf = \Rb^d$, a non-Euclidean space or a more abstract space. 
For any probability measure $\P$ and $Q$ on measure space  $(\Xf,\Xc)$ with densities respectively $p$ and $q$  with respect to    some base measure $\lambda$, the \myredvsec{total} variation distance between them is
	$V(\P,Q) = \sup_{A\in \Xc} |\P(A)-Q(A)| = \frac{1}{2} \int_{\Xf}  |p(x)-q(x)|d\lambda$. 
The Hellinger distance is given by
	 $ h(\P,\Q) = \left(\int_{\Xfrak} \frac{1}{2} |\sqrt{p(x)}-\sqrt{q(x)}|^2d\lambda\right)^{\frac{1}{2}}$.
	The Kullback-Leibler divergence of \myredvsec{$P$ from $Q$}  is $K(p,q)=\int_{\Xfrak}p(x)\log\frac{p(x)}{q(x)} d\lambda$.

%

The vector of all zeros is denoted as
$\bm{0}$ (in bold). Any vector $x\in \Rb^d$ is a column vector with its $i$-th coordinate denoted by $x^{(i)}$. 
%
Denote the set $[k]:=\{1,2,\ldots,k\}$. 
Denote the Dirac measure at $\theta$ as $\delta_{\theta}$.   
%
In the presentation of asymptotic inequality bounds, write $a \lesssim b$ if $a\leq C b$ for some constant $C$, where the dependence of $C$ on other quantities will be spelled out in all key theorems. 
\myredvsec{In particular, the constant may depend on a
true, but unknown probability measure in some cases.}
$a\asymp b$ if $a\lesssim b$ and $b\lesssim a$.

\myredvsec{Denote by $C(\cdot)$ or $c(\cdot)$  a positive finite constant depending only on its parameters and the probability kernel $\{\Pb_\theta\}_{\theta\in \Theta}$. In the presentation of inequality bounds and proofs,  they may differ from line to line. }



\section{Inverse bounds for finite discrete measures}
\label{sec:finite}

The focus of this section is a general setting of finite mixture models, which specifies the law for observed sample $\Dset
\subset \Xf$ by the probability measure $P_G( \textrm{d}x) = \int_{\Theta} \Pb(\textrm{d}x|\theta) G(\textrm{d}\theta)$. The object of inferential interest is the mixing measure $G$, a discrete probability measure with a finite number of support points. Here, $f$ denotes a known probability kernel on $\Xf\times \Theta$.
Assuming that $\Pb(\textrm{d}x|\theta)$ is absolutely continuous with respect to a dominating measure $\lambda$ on $(\Xf,\Xc)$, 
we use $f(x|\theta)$ to denote its density function, and $p_G$ the corresponding density function for $P_G$, as given in Eq.~\eqref{Eq:mixture}.

\subsection{Test functions}
\label{sec:test}
The question of identifiability of the latent $G(\textrm{d}\theta)$ from the observed mixture distribution $P_G(\textrm{d} x)$ may be viewed through the lens of test functions on $\Theta$. In particular, the functions that arise from the kernel $f$, namely, the class $\{ f(x\mid\cdot) \}_{x\in\Xf}$, as $x$ varies in $\Xf$,
represents a collection of test functions for the probability measures on $\Theta$. In general, given a class $\Phi$ of real measurable functions on $\Theta$, we are interested in identifying probability measures $G$, which may be viewed as linear functionals on $\Phi$: $\phi \mapsto G\phi$,  where $G\phi:= \int \phi(\theta) G(\textrm{d}\theta)$. Clearly, $G$ is identifiable in the topology of weak convergence by $\Phi$, if $\Phi$ is sufficiently rich, e.g., if $\Phi$ contains all continuous and bounded functions on $\Theta$. Our interest here is to exploit the specific constraint that $G$ is a discrete measure that has a bounded number of support points. This motivates the following.

\begin{comment}
For a measurable function $\myf$ defined on $\Theta$, its integral w.r.t. a distribution $G=\sum_{i=1}^k p_i \delta_{\theta_i} \in \Ec_k(\Theta)$ is denoted by $G\myf:=\int \myf dG = \sum_{i=1}^k p_i \myf(\theta_i)$. The notation $G\myf$ is used to emphasize that $G$ can be viewed as an linear operator on measurable functions on $\Theta$. Consider a family $\myF$ of real-valued functions defined on $\Theta$. Given $\myF$,  $\|G-H\|_{\Phi}:=\sup_{\myf\in \myF} |G \myf-H \myf|$ measure the deviation between two mixing measures $G$ and $H$. A natural requirement for the test functions is the following property. 
\end{comment}

\begin{definition}
\label{def:distin}
Let $\Gc_k(\Theta)$ denote the set of discrete probability measures with at most $k$ atoms on $\Theta$.
	$\Gc_k(\Theta)$ is distinguishable by $\myF$ if for any $G \neq H\in \Gc_k(\Theta)$, 
    $\|G-H\|_{\Phi}:= \sup_{\phi \in \Phi} |G \phi - H \phi| >0$. 
\end{definition}
It is straightforward that if $\Gc_k(\Theta)$ is distinguishable by $\myF$, then  $\|G-H\|_{\Phi}$ is a proper metric on $\Gc_k(\Theta)$, which shall be referred to from here on as a $\Phi$-distance \citep{wei2023minimum}. Remarkably, it will be shown in the sequel that many inverse bounds obtained in the literature for finite mixture models may be expressed in the form of the inequality
\begin{equation}
\label{Eq:inversePhi}
    W_r^r(G,H) \lesssim \|G-H\|_{\Phi},
\end{equation}
for some $r\geq 1$ and for some suitable choice of function class $\Phi$. It is informative to first consider several examples of such function classes. \myredvsec{Note that when $\Theta$ is bounded, it is immediate from the definition \eqref{eq:optdist} that $W_r^r(G,H) \lesssim W_s^s(G,H)$ if $r>s \geq 1$. }

\begin{comment}
\begin{exa}[Moment deviations between mixing distributions] \label{exa:moment}
Suppose $q=1$ for simplicity.  Consider $\myF_2=\{(\theta-\theta_0)^j\}_{j\in [2d_1-1]}$ to be a finite collection of polynomials with $\theta_0$ a fixed constant. 
Then 
\begin{equation}
 \|G -H \|_{\myF_2} = \|\mbf_{2d_1-1}(G-\theta_0)-\mbf_{2d_1-1}(H-\theta_0)\|_{\infty}, \label{eqn:momentdistance}
\end{equation}
which is the maximum deviation of the first $2d_1-1$ moments of $G-\theta_0$ and $H-\theta_0$. Note that one may also include the index $j=0$ in the definition of $\myF_2$. Abusing the notation, we also denote:
\begin{align*}
& \mbf_{2d_1-1}(G-\theta_0,H-\theta_0) \\
:=&
\|\mbf_{2d_1-1}(G-\theta_0)-\mbf_{2d_1-1}(H-\theta_0)\|_{\infty}. 
\end{align*}
This example  can be naturally extended to general $q$-dimensional $\Theta$. 
\myeoe
\end{exa}
\end{comment}

\begin{exa} \label{exa:moment}
Firstly we  introduce some notation on multi-index and moments. 
The multi-index notation for $\alpha\in \Nb^q$
imposes the following 
$$
|\alpha|:=\sum_{i\in [q]} |\alpha^{(i)}|,  \quad \theta^\alpha :=  \prod_{i\in [q]} \left(\theta^{(i)}\right)^{\alpha^{(i)}},
$$
\myred{where $\theta\in \Rb^q$.} Denote $\Ic_k:=\{\alpha\in \Nb^q \mid  |\alpha|\leq k \}$. For a finite signed (discrete) measure $G=\sum_{i\in [k]} {p_i}\delta_{\theta_i}$ on $\Rb^q$, its $\alpha$-th moment is $m_\alpha(G)= \int \theta^{\alpha} \textrm{d}G(\theta) = \sum_{i\in [k]} {p_i}\theta_i^\alpha\in \Rb^q$. Denote by $\mbf_k(G) :=(m_\alpha(G))_{\alpha\in \Ic_k}\in \Rb^{|\Ic_k|}$ the vector of all $\alpha$-th moments of $G$ for $\alpha\in \Ic_k$. 

Suppose $\Theta = \Rb^q$, $\theta_0$ is a fixed constant vector and let $\Phi= \{(\theta-\theta_0)^{\alpha}\}_{|\alpha|\leq 2k-1}$, a finite collection of monomials. 
Then 
$$\|G-H\|_\Phi = \|\mbf_k(G-\theta_0)-\mbf_k(H-\theta_0) \|_\infty$$ 
 the infinity norm of a vector space of moments of degree up to $2k-1$. Here $G-\theta_0:=\sum_{i\in [k]} {p_i}\delta_{\theta_i-\theta_0}$ is the discrete distributions obtained by shifting the support points of $G$ by $-\theta_0$. 
\myeoe
\end{exa}

The next example is rather interesting, because the class of test functions is composed of the kernel function $f$ that defines the mixture model.

\begin{exa}
\label{exa:divationbetweenmixturedistributions}
Let $\myF=\left\{ \theta \mapsto \int h(x) f(x|\theta) \dd \lambda(x)| h \in \Hc \right\}$ where $\Hc$ is a class of bounded and measurable functions on $(\Xf,\Xc)$. Note that each $h \in \Hc$ defines a function $\phi$ such that $\phi(\theta) = \int h(x) f(x|\theta) \dd \lambda(x)$. By Fubini's theorem, 
\begin{equation}
\label{eqn:Gphi}
G\phi = \int\int h(x)f(x|\theta) \dd \lambda(x) G(\dd \theta) = \int h(x) P_G(\dd x).
\end{equation}
It follows that
\begin{equation}
\|G-H\|_{\Phi} = \sup_{h \in \Hc} \left|\int h(x) (\Pb_G - \Pb_H) (\dd x)\right|, \label{eqn:integraldistance}
\end{equation}
which is also known as the integral probability metrics (IPM) \citep{muller1997integral, sriperumbudur2012empirical}, which are defined on the space of induced data population's distributions $\{P_G\}$.  

Let us specialize this example further. Taking $\Hc=\left\{x \mapsto 1_{B}(x)|B\in \Xc \right\}$, where $1_B(x)$ is the indicator function taking the value $1$ when $x\in B$ and $0$ otherwise, $\eqref{eqn:integraldistance}$ reduces to the total variation distance $V(\Pb_G,\Pb_H)$. Thus, an inverse bound of the form \eqref{Eq:inversePhi} reduces to the earlier form given by \eqref{Eq:inverse}.
For $(\Xf,\Xc)=(\Rb,\Bc(\Rb))$, the real line endowed with the Borel sigma algebra, and take $\Hc= \left\{x \mapsto 1_{(-\infty,a]}(x)|a\in \Rb \right\}$,  $\eqref{eqn:integraldistance}$ becomes the Kolmogorov-Smirnov (KS) distance $\KS(\Pb_G,\Pb_H)$, which is the maximum deviation of the cumulative distribution functions (CDF) of the mixture distributions.  By varying the choice of $\Hcal$, we may obtain a variety of metrics that drive the establishment of inverse bounds, both old and new, including the case where the domain $\Xf$ is non-Euclidean.
%
 \myeoe
\end{exa}

\subsection{Pointwise and uniform inverse bounds}

As illustrated in Section \ref{sec:inverseintro}, inverse bounds enable the derivation of rates of convergence for the unobserved quantity of interest $G$.  There are two types of inequalities to be discussed, which will be known as \emph{pointwise} and \emph{uniform} inverse bounds. They are useful in assessing the convergence of an estimate of the mixing measure $G$ in either a pointwise or a minimax sense. It must be noted that for latent structured models in general and finite mixture models in particular, the optimal minimax estimation rate is typically different (and considerably slower) than the optimal pointwise estimation rate for the mixing measure, due to the highly nonlinear nature of the desirable estimators involved.  However, both types of optimality are meaningful, depending on the specific application domain. To fully assess the quality of a proposed estimation procedure for the mixture models and to speak of optimality one needs to consider both types of convergence, for which both pointwise and uniform inverse bounds will play a fundamental role. 

\begin{comment}
A standard way for characterizing the difficulty of an estimation problem is via minimax 
lower bounds for the quantity of interest. An estimation procedure is then evaluated against this metric of performance; the procedure is considered optimal in the \emph{minimax sense} if the corresponding minimax estimation upper bound guarantee matches the minimax lower bound under the same setting. It must be noted that for mixture models, the optimal minimax estimation rate is typically much slower than the optimal \emph{pointwise} estimation rate for the mixing measure. Thus, in theory a ``minimax optimal" procedure is not necessarily optimal in the sense of pointwise convergence, and vice versa. To fully assess the quality of a proposed estimation procedure, in this paper we will characterize the proposed estimation procedure using both types of convergence bounds.
\end{comment}

Considerable progress have been made in the past decade regarding optimal rates of convergence for mixing measure estimation
for the finite mixture setting. The primary setting is that of "overfitted" mixture models, where the true number of mixture components is unknown, only an upper bound $k$ is given. That is, $G\in \Gc_k(\Theta)$. The optimal minimax rate for estimating the mixing measure \myred{for general mixture models} was first established by \cite{heinrich2018strong} in the univariate parameter setting, i.e., when $\Theta \subset \Rb$, and univariate data space $\Xf=\Rb$. Prior to \cite{heinrich2018strong}, there were also related minimax results \citep{Hardt2015-rg, Kalai2010-jn, chen2003tests} specifically for Gaussian mixture models, i.e., $f$ is a normal kernel. On the other hand, the optimal pointwise convergence rate of convergence for the mixing measure in the overfitted mixture setting is the parametric rate $n^{-1/2}$ under quite general settings. 
Various estimation methods have been shown to achieve this rate of pointwise convergence (possibly up to a logarithm factor) \citep{heinrich2018strong,ho2020robust,guha2021posterior}. 

In this subsection we will introduce uniform and pointwise inverse bounds, which are essential ingredients underlying the aforementioned pointwise and minimax theory. Our intention is to clarify the roles of these structural results and their consequences. The establishment of such inverse bounds is deferred to Section \ref{sec:identifiability}.

\vspace{.1in}
\noindent \underline{Uniform inverse bounds}
All inverse bounds in this section are concerned with finite discrete measures, subject to constraints on the number of supporting atoms, and will be described in terms of the $\Phi$-distances introduced in Section \ref{sec:test}. Hence they may be viewed as a property of the ambient space $\Gc_k$, for a given $k \in \Nb$, and the test function class $\Phi$. 

The pair $(\Gc_k,\Phi)$ is said to possess the \emph{global inverse bound} if the following holds:
\begin{equation}
	\inf_{  G\neq H\in \Gc_k(\Theta) } \frac{\|G-H\|_{\Phi}}{W_{2k-1}^{2k-1}(G,H)} >0. \label{eqn:globalinversebound}
\end{equation}
Also useful is a local version of the uniform inverse bound, around a neighborhood of some $G_0 \in \Gc_k(\Theta)$. In particular suppose that $G_0$ has exactly $k_0 \leq k$ supporting atoms, we say the triplet $(\Gc_k,\Phi,G_0)$ possesses the \emph{local inverse bound} if the following holds:
\begin{equation}
\liminf_{ \substack{G,H\overset{W_1}{\to} G_0\\ G\neq H\in \Gc_k(\Theta) }} \frac{\|G-H\|_{\Phi}}{W_{2d_1-1}^{2d_1-1}(G,H)} >0. \label{eqn:localinversebound}
\end{equation}
In the above equality, one may replace the RHS by a positive constant $C$, which depends on $G_0$. In addition, $d_1$ is a function of $G_0$: each $G_0$ has a unique number of atoms $k_0$, and thus has a unique overfit index $d_1=k-k_0+1$. Hence, this bound is local as it holds only for a small neighborhood of $G_0$, with a constant $C$ which also depends on $G_0$.
\myredvsec{Note that $d_1\leq k$, and thus one improves the denominator from $W_{2k-1}^{2k-1}(G,H)$ in the global inverse bound to $W_{2d_1-1}^{2d_1-1}(G,H)$ in the local inverse bounds by restricting the domain to be around some $G_0$.} 
It is worth noting that the two notions of uniform inverse bounds are closely related via a simple compactness argument \citep{wei2023minimum}:
\begin{lem} \label{lem:localtoglobal}
Suppose that $\Theta$ is compact, and $\Gc_k(\Theta)$ is distinguishable by $\myF$.
	If the local inverse bound \eqref{eqn:localinversebound} holds for any $G_0\in \Gc_k(\Theta)$, then so does the global version \eqref{eqn:globalinversebound}.
\end{lem}

Clearly if the global (or local) inverse bound is possessed by $(\Gc_k, \Phi)$ (or by $(\Gc_k, \Phi, G_0)$, resp.), then it also holds if $\Phi$ is replaced by any arbitrary superset. Thus it is of interest to establish such bounds for minimal classses of test functions $\Phi$.  Such uniform inverse bounds have been studied in \cite{heinrich2018strong, wu2020optimal, doss2020optimal} for specific  $\myF$ and \cite{wei2023minimum} for general $\myF$. 

\begin{comment}
The following lemma states some equivalent formulations of inverse bounds. 

\begin{lem}[Equivalent versions of inverse bounds]The following hold. 
\label{lem:equinvbou}
 \begin{enumerate}[label=(\alph*)]
\item \label{lem:equinvboua}
\eqref{eqn:globalinversebound} is equivalent to
\begin{align*}
	 W_{2k-1}^{2k-1}(G,H) \leq C' \|G-H\|_{\Phi}, \quad \forall G, H\in \Gc_k(\Theta)
\end{align*}
for some constant $C'$ (that possibly depends on the model). 

\item  \label{lem:equinvboub}
 Fix $G_0\in \Ec_{k_0}(\Theta)$. \eqref{eqn:localinversebound} is equivalent to the following:
 there exist $r(G_0)$ and  $C(G_0)$, where their dependence on $\myF, \Theta, k_0,k$ are suppressed, such that for any $G,H\in \Gc_k(\Theta)$ satisfying $W_1(G_0,G)<r(G_0)$ and $W_1(G_0,H)<r(G_0)$, 
\begin{align*}
 W_{2d_1-1}^{2d_1-1}(G,H)  
\leq &  C(G_0)   \|G-H\|_{\Phi}.
\end{align*}

\item \label{lem:equinvbouc}
Suppose that $\Theta$ is compact and that $\Gc_k(\Theta)$ is distinguishable by $\myF$.  Fix $G_0\in \Ec_{k_0}(\Theta)$.  \eqref{eqn:localinversebound} is equivalent to the following: there exist $r(G_0)$ and  $C(G_0)$, where their dependence on $\myF, \Theta, k_0,k$ are suppressed, such that for any $ G, H\in \Gc_k(\Theta)$ satisfying $W_1(G_0,H)<r(G_0)$, 
\begin{align*}
 W_{2d_1-1}^{2d_1-1}(G,H)
\leq &  C(G_0)     \|G-H\|_{\Phi}.
\end{align*}
\end{enumerate}
\end{lem}
\end{comment}

\vspace{.1in}
\noindent\underline{Pointwise inverse bounds} 
\label{sec:invboufix}
If one only intends to establish a pointwise convergence rate for a particular true mixing measure, it   suffices to obtain an inverse bound with one argument fixed. As before, suppose that $G_0$ has exactly $k_0$ atoms, the triplet $(\Gc_k, \Phi, G_0)$ is said to possess the \emph{pointwise overfitted inverse bound} if there holds
\begin{equation}
\liminf_{ \substack{G\overset{W_1}{\to} G_0\\ G\in \Gc_k(\Theta) }} \frac{\|G-G_0 \|_{\myF}}{W_{2}^{2}(G,G_0)} >0. \label{eqn:localinverseboundfixargument}
\end{equation}
An useful refinement of the above inequality can be achieved, if we restrict the ambient space to $\Gc_{k_0} \subset \Gc_{k}$. In particular, the triplet $(\Gc_{k_0},\Phi, G_0)$ is said to possess the \emph{pointwise exact-fitted inverse bound} if
\begin{equation}
\liminf_{ \substack{G\overset{W_1}{\to} G_0\\ G\in \Gc_{k_0}(\Theta) }} \frac{\|G-G_0 \|_{\myF}}{W_{1}(G,G_0)} >0. \label{eqn:localinverseboundfixargumentexactfitted}
\end{equation}

Pointwise overfitted inverse bounds and various inequalities in similar forms have been obtained in \cite{chen1995optimal,nguyen2013convergence,ho2016strong,wei2023minimum}. 
The pointwise exact-fitted inverse bounds were studied in \cite{ho2016strong,wei2022convergence, wei2023minimum}.  
They play crucial roles in deriving pointwise rates of convergence for various estimators of mixing measures that arise in overfitted and exact-fitted finite mixture models.

\begin{comment}
Pointwise inverse bounds may be used to derive pointwise convergence rates. 
A general pointwise convergence rate in the spirit of Lemma \ref{lem:convergencerate} for arbitrary estimators maybe formulated but is omitted due to the space constraint. We refer to \cite{chen1995optimal,nguyen2013convergence,ho2016strong,wei2023minimum,wei2022convergence} for various specific function classes $\myF$ and 
\cite[Section 4]{wei2023minimum} for general $\myF$. 
\end{comment}

\begin{rem} \label{rem:strweakiden}
As noted earlier, all inverse bounds \eqref{eqn:globalinversebound}, \eqref{eqn:localinversebound}, \eqref{eqn:localinverseboundfixargument} and \eqref{eqn:localinverseboundfixargumentexactfitted} may be written in the form of Eq. \eqref{Eq:inversePhi}, which expresses an upper bound of $W_r^{r}(G,H)$ for some $r\geq 1$ in terms of the $\Phi$-distance $\|G-H\|_{\Phi}$, provided that $W_r(G,H)$ is sufficiently small.
An inquisitive reader will likely question the seeming arbitrariness in the order $r$ in these inequalities. It is the technical subject of Section \ref{sec:identifiability}, where we shall elaborate general sufficient conditions on the test function class $\Phi$ under which the order $r$ introduced herein for the relevant inverse bounds are, in fact, the best possible. Such sufficient conditions shall be referred to as \emph{strong identifiability} conditions. However, when such conditions on $\Phi$ are violated, e.g., when we are in the settings of singular mixture models, then a generic choice for the order $r$ as stated may not longer be valid. Remarkably, useful pointwise inverse bounds may still be obtained with other appropriate choices of order $r$. Such results will be found under the banner of \emph{weak identifiability} theory and discussed in Sections \ref{sec:weakidentifiability} and \ref{sec:singular}. \myeoe
\end{rem}
\begin{rem}[Inverse bounds via moment differences]
Despite the versatility of $\Phi$-distances, it is worth mentioning that related  and alternative characterizations of the inverse bounds are available. Suppose that the function class $\Phi$ contains only continuously differentiable functions up to $2k-1$ order satisfying a suitable regularity condition, then it can be shown \cite[Section 2.5]{wei2023minimum} that
\[\|G-H\|_{\myF} \lesssim \|\mbf_{2k-1}\left(G-\theta_0\right)-\mbf_{2k-1}\left(H-\theta_0\right) \|_\infty,\] 
where \myredvsec{$\theta_0$ is a fixed constant vector} and the right side represents the infinity norm of the vector of moment differences, for all moments up to degree $2k-1$
(i.e., a specific $\Phi$-distance with monomial test functions, cf. Example \ref{exa:moment}).
Moreover, when $\Phi$ satisfies the strong identifiability conditions that make \eqref{eqn:globalinversebound} valid, then the above asymptotical inequality can be established in the reverse order, so we arrive at the following inverse bound in terms of the moment differences:
\begin{multline}
W_{2k-1}^{2k-1}(G,H)  \lesssim \|G-H\|_{\myF} \\
\asymp  \|\mbf_{2k-1}\left(G-\theta_0\right)-\mbf_{2k-1}\left(H-\theta_0\right) \|_\infty.
\label{eq:moments}
 \end{multline}
Thus we recover an inverse bound of \cite{doss2020optimal} (who studied Gaussian mixtures), but in a general mixture setting.
\myeoe
\end{rem}

\begin{rem}[Inverse bounds via integral $\Phi$-distances] 
\label{rem:invbouint}
All inverse bounds \eqref{eqn:globalinversebound}, \eqref{eqn:localinversebound}, \eqref{eqn:localinverseboundfixargument}, \eqref{eqn:localinverseboundfixargumentexactfitted}, or in its generic form of Eq.\eqref{Eq:inversePhi} state the upper bound of optimal transport distances $W_r(G,H)$ in terms of the $\Phi$-distance, which corresponds to the infinity norm of the class of linear functionals $G: \phi \mapsto G\phi$. Alternatively, suppose that there is a measure $\mathscr{T}$ on $\myF$, and one may also be interested in inverse bounds in terms of an "integral $\Phi$-distance":
\begin{equation}
\label{eq:inverseintegral}
W_r^r(G,H) \lesssim \int_{\myF} |G \myf-H \myf| d\mathscr{T}. 
\end{equation}
Take a class of functions that arise from the kernel $f$ of the mixture model, namely, $\Phi = \{f(x|\cdot)\}_{x \in \Xf}$, and $\mathscr{T}$ the measure induced from the dominating measure $\lambda$ on $(\Xf, \Xc)$, then clearly
\[\int_{\myF} |G \myf-H \myf| d\mathscr{T} = \int_{\Xf} |p_G(x) - p_H(x)| d\lambda = 2 V(P_G, P_H).\]
Thus, the inverse bound \eqref{eq:inverseintegral} reduces to the form of \eqref{Eq:inverse}. 
It turns out that much of the machinery that goes into the establishment of the inverse bounds via $\Phi$-distances can be extended to that of integral distances. See \cite[Section~5.1]{wei2023minimum} for further details. 
\myeoe
\end{rem}

\begin{comment}
\begin{rem} \footnote{require careful check}We present a summary remark to conclude this subsection. 
\begin{itemize}
  \item The richer $\Phi$ is, the more ``likely'' the inverse bounds hold.
  
  \item As seen from Lemma \ref{lem:convergencerate}, the bounds are \emph{estimation--procedure--free}, i.e., they can be used to derive convergence rates 
  based on the convergence behavior of $\|G-\hat{G}_n\|_{\myF}$
  for a given estimation procedure (Bayes, MLE, etc.).

  \item They suggest a large class of estimation methods as well.

  \item This theory and related methods can be found in detail in \cite{wei2023minimum}.
\end{itemize}
\myeoe
\end{rem}
\end{comment}

\noindent\underline{Consequences of inverse bounds}
Having introduced the relevant inverse bounds, we shall now make explicit their usefulness in deriving convergence behavior for an estimate of the mixing measure. The following is a simple but rather meaningful lemma which spells out precisely the role of uniform inverse bounds \eqref{eqn:globalinversebound} and \eqref{eqn:localinversebound}. 
Recall that $\Pc(\Theta)$ is the space of all Borel probability measure on $\Theta$. Denote by $\Ef_n$ the set of all estimators (measurable random elements) taking values in $\Pc(\Theta)$ 
based on an $n$-i.i.d. sample $X_1,\ldots,X_n$ from the mixture distribution $\Pb_{G^*}$.

\begin{lem}
\label{lem:convergencerate}
Suppose that $\Theta$ is compact. Let $\hat{G}_n\in \Ef_n$ be any estimator.
 \begin{enumerate}[label=(\alph*)]
\item \label{itema:lem:convergencerate}
Suppose that $(\Gc_k,\Phi)$ possesses the global inverse bound \eqref{eqn:globalinversebound}. 
Then for any $G\in \Gc_k(\Theta)$, any $t>0$, there is a constant $C(\Phi,\Theta,k)>0$ such that
\begin{align*}
\numberthis \label{eqn:globalconvergencerategenerallem}
	& \left\{W_{2k-1}^{2k-1}(G,\hat{G}_n)\geq t \right\} \cap \{ \hat{G}_n \in \Gc_k(\Theta) \} \\
	\subset &\left\{ \|G -\hat{G}_n \|_{\myF}\geq C(\myF,\Theta,k) t \right\} \cap \{ \hat{G}_n \in \Gc_k(\Theta) \}. 
\end{align*}
\item \label{itemb:lem:convergencerate}
Fix $G_0\in \Gc_k(\Theta)$ that has exactly $k_0\leq k$ atoms. Suppose that $(\Gc_k, \Phi, G_0)$ possesses the local inverse bound \eqref{eqn:localinversebound} and $\Gc_k(\Theta)$ is distinguishable by $\myF$. Then there exist positive constants $r(G_0)$ and  $C(G_0)$, whose dependence on $\myF, \Theta, k_0,k$ are suppressed, such that for any $G\in \Gc_k(\Theta)$ satisfying $W_1(G_0,G)<r(G_0)$, any $t >0$,
\begin{align*}
\numberthis \label{eqn:convergencerategenerallem} 
&\left\{ W_{2d_1-1}^{2d_1-1}(G,\hat{G}_n)\geq t \right\} \cap \{ \hat{G}_n \in \Gc_k(\Theta) \} \\
\subset &    \left\{ \|G -\hat{G}_n \|_{\myF}\geq C(G_0)  t  \right\} \cap \{ \hat{G}_n \in \Gc_k(\Theta) \} . 
\end{align*}
\end{enumerate}
\end{lem}

Although there is nothing to prove about \eqref{eqn:globalconvergencerategenerallem}, as it follows directly from the hypothesis that \eqref{eqn:globalinversebound} holds, there are some important interpretations (we omit the interpretations of
\eqref{eqn:convergencerategenerallem} due to the similarity).
Equation \eqref{eqn:globalconvergencerategenerallem} states that the event on $(\Xf,\Xc)$ defined in terms of a Wasserstein distance between mixing measures is a subset of the event defined in terms of $\|G -\hat{G}_n \|_{\myF}$. To quantify the convergence rate in the Wasserstein distance, it suffices to have a control on $\|G -\hat{G}_n \|_{\myF}$. 

Since \eqref{eqn:globalconvergencerategenerallem} is a relationship of events on the probability space $(\Xf,\Xc)$, we may evaluate the events under any probability measure to obtain a upper bound of the tail probability $\Pb\left(W^{2k-1}_{2k-1}(G,\hat{G}_n)\geq t\right)$. The natural probability measure is the true $\Pb_{G^*}$ under which one obtain the observed i.i.d. samples $X_1,\ldots,X_n\overset{\iidtext}{\sim} \Pb_{G^*}$.
Another choice for the probability measure  is a posterior distribution $\Pi(\cdot|X_1,\cdots,X_n)$, which is derived via Bayes' formula from some prior distribution $\Pi$ 
on the space of mixing measures (see \cite{nguyen2013convergence,wei2022convergence}). Convergence rates of a variety of estimators $\hat{G}_n$ have been obtained this way in the literature \citep{ho2016strong,ho2016convergence,heinrich2018strong,wu2020optimal,doss2020optimal,wei2022convergence}, ranging from Bayesian estimators, maximum likelihood estimators, (MLE) to minimum distance and moment based methods. Claims such as Eq.\eqref{eqn:globalconvergencerategenerallem} and Eq.\eqref{eqn:convergencerategenerallem} allow one to transfer any probability analysis of the tail event expressed in terms of the $\Phi$-distances into that of a Wasserstein distance. 


\begin{comment}
\begin{rem} Conditions involved include
\begin{itemize}
    \item $\Theta$ being compact.
    \item The collection of $\phi \in \Phi$ is differentiable up to certain order depending on $k$, 
    and satisfies a suitable linear independence property.
    \item These inverse bounds are useful for establishing local minimax bounds on estimators of $G^*$.
\end{itemize}
\myeoe
\end{rem}
\end{comment}

\subsection{Minimum distance estimators}
\label{sec:mindistance}
We have explained how inverse bounds of the forms \eqref{Eq:inversePhi}, \eqref{eq:moments} or \eqref{eq:inverseintegral} are "estimator-free" structural results that can be used to derive rates of convergence of the mixing measure for any estimator in consideration, provided that one can obtain a convergence analysis for the right hand side of such bounds. For the remainder of this section we shall demonstrate that such inverse bounds are also useful in deriving concrete estimation procedures for the mixing measure. In fact, almost all M-estimation procedures investigated in the existing literature on mixture models may be presented in this way --- as a minimization procedure of the right hand side of an inverse bound. More precisely, they correspond to the minimization of an empirical version of a distance on the space of data population's distributions. Now, we will specifically present a general class of minimum $\Phi$-distance estimator, which was formulated by \cite{wei2023minimum}, although a similar approach may be described for the integral $\Phi$-distance as well. We will then point out that the minimum $\Phi$-distance estimator reduces to existing methods by picking out a specific class of test functions $\Phi$, but it also leads to novel methods that are suitable for more complex data domains. 

The general idea of a minimum distance estimator for parametric models is intuitive and well-known, where the distance in question may be Hellinger distance \citep{beran1977}, or KL divergence \citep{vandeGeer-00}, and so on. Similarly, to make use of a $\Phi$-distance to obtain an estimate for $G$, given an $n$-i.i.d. sample $\Dset = \{X_1,\ldots,X_n\}$, which is distributed according to the true mixture distribution $P_{G^*}$, one might consider to minimize a $\Phi$-distance between $P_{G}$ and an empirical estimate of $P_{G^*}$. There is a non-trivial issue, however. Note that the $\Phi$-distance $\|G-G^*\|_{\Phi} = \sup_{\phi\in \Phi} |G\phi - G^*\phi|$, which in this general form defines a distance on the space of mixing measures, not the space of mixture distributions. Thus, in order to employ this distance for the estimation purpose, the quantity $G^*\phi$ must be estimated from the empirical sample $\Dset$. Since $G^*$ is unknown, one should be prepared to estimate $G\phi$ for all $G\in \Gc_{k}$ and $\phi \in \Phi$. This motivates the following definition.

%

\begin{definition}
\label{def:estimable}
The family $\myF$ is said to be \emph{estimable} on $\Gc_{k}(\Theta)$ if
for each $\myf\in \myF$, there exists a measurable function $t_\phi$ defined on $\Xf$ such that 
$G \myf=\int_{\Xf} t_\myf(x) P_{G}(\dd x)$ for any $G\in \Gc_k(\Theta)$. In other words, the statistic $t_\myf(X_1)$ is an unbiased estimate for $G\myf$.    
\end{definition}
Clearly, if $\myF$ is estimable 
then $G^* \myf = \Eb t_\myf(X_1)$, where the expectation is under the true distribution $P_{G^*}$. Thus, $G^*\phi$ admits an unbiased estimate via $\frac{1}{n} \sum_{i\in [n]} t_\myf(X_i)$. 

A general example for estimable function class $\Phi$ can be found in Example \ref{exa:divationbetweenmixturedistributions}: for a function $\phi \in \Phi$ which may be expressed as $\phi(\theta) = \int h(x)f(x|\theta) \dd \lambda(x)$, due to \eqref{eqn:Gphi}, the statistic $h(X_1)$ yields an unbiased estimate for $G^*\phi$, as required in the above definition.

\begin{comment}
We say the finite mixture model $\Pb_G$ is \emph{identifiable} on $G_k(\Theta)$ if $G\mapsto \Pb_{G}$ is injective on $\Gc_k(\Theta)$. The next lemma is a straightforward result connecting several definitions. 

\begin{lem}
    If $\myF$ is estimable on $\Gc_{k}(\Theta)$ and $\Gc_k(\Theta)$ is distinguishable by $\myF$, then the mixture model is identifiable on $\Gc_k(\Theta)$. \myeoe
\end{lem}
\end{comment}

Provided that $\myF$ is estimable on $\Gc_{k}(\Theta)$, and assume that $\Theta$ is compact, we obtain an estimate for mixing measures by solving
$$
\hat{G}_n(\ell)\in \argmin_{G\in \Gc_\ell(\Theta)} \sup_{\myf\in\myF}\left|G \myf-\frac{1}{n}\sum_{i\in [n]}t_\myf(X_i)\right|, \quad \forall \ell\geq 1. 
$$ 
 Note that $\hat{G}_n(\ell)$ is well-defined since 
 the objective function of the above minimization problem
 is non-negative 
 and lower semicontinuous with respect to $G$, which implies that its minimum is attained on the compact space $\Gc_\ell(\Theta)$. Given that $G^*$ has at most $k$ atoms, a natural choice for an estimate of $G^*$  would be $\hat{G}_n:=\hat{G}_n(k)$, which is termed a \emph{minimum $\myF$-distance estimator}.

It is easy to observe that for any $G\in \Gc_k(\Theta)$, 
\begin{eqnarray*}
& \|\hat{G}_n-G \|_{\myF}  =  \sup_{\phi\in \Phi} |\hat{G}_n \phi - G\phi| \leq \\
& 
\sup_{\phi\in \Phi} \left|\frac{1}{n}\sum_{i\in [n]}t_\myf(X_i)- \hat{G}_n \myf \right| 
+  \left|\frac{1}{n}\sum_{i\in [n]}t_\myf(X_i)- G \myf \right|  \\
& \leq 
2\sup_{\myf\in \myF} \left|\frac{1}{n}\sum_{i\in [n]}t_\myf(X_i)- G \myf \right|, \label{eqn:Gnproperty}
\end{eqnarray*}
where the first inequality in the above display is due to the triangle inequality and the second inequality follows from the definition of $\hat{G}_n$. Thus, the $\Phi$-distance between the estimate $\hat{G}_n$ and $G$ may be controlled by the suprema of an empirical process that is associated with the function class $\Phi$ and probability measure $G$. Combining this inequality with Lemma \ref{lem:convergencerate} by putting $G=G^*$, we immediately obtain the following.


\begin{thm}
\label{thm:convergencerate}
Suppose that $\Theta$ is compact, $\Gc_k(\Theta)$ is distinguishable by $\myF$, and that $\myF$ is estimable on $\Gc_{k}(\Theta)$. Let $\hat{G}_n$ be a minimum $\myF$-distance estimator. 
 \begin{enumerate}[label=(\alph*)]
\item 
Assume that $(\Gc_k,\Phi)$ possesses the global inverse bound \eqref{eqn:globalinversebound}. 
Then there is a positive constant  $C$, whose dependence on $ \Theta, k, \myF$ is suppressed, such that for any $G^*\in \Gc_k(\Theta)$, any $t>0$, 
\begin{align*}
\numberthis \label{eqn:globalconvergencerategeneralpro}
	& \Pb_{G^*} \left(W_{2k-1}^{2k-1}(G^*,\hat{G}_n)\geq t \right) \\
	& \leq \Pb_{G^*} \left( \sup_{\myf\in \myF} \left|\frac{1}{n}\sum_{i\in [n]}t_\myf(X_i)- G^* \myf \right| \geq C  t  \right), 
\end{align*} 
and
\begin{align*}
	& \Eb_{G^*} W_{2k-1}^{2k-1}(G^*,\hat{G}_n) \\
    & \leq C \, \Eb_{G^*}  \sup_{\myf\in \myF} \left|\frac{1}{n}\sum_{i\in [n]}t_\myf(X_i)- G^* \myf \right| . 
\end{align*}

\item \label{itemb:thm:convergencerate}
Suppose that $G_0$ has exactly $k_0\leq k$ atoms in $\Theta$ and that $(\Gc_k,\Phi,G_0)$ possesses the local inverse bound \eqref{eqn:localinversebound}.  Then there exist positive constants $r(G_0)$, $C(G_0)$ and $c(G_0)$, whose dependence on $ \Theta, k_0,k,\myF$ are suppressed, such that for any $G^*\in \Gc_k(\Theta)$ satisfying $W_1(G_0,{G^*})<r(G_0)$, 
\begin{align*}
\numberthis \label{eqn:convergencerategeneral1}
&\Pb_{G^*} \left(W_{2d_1-1}^{2d_1-1}(G^*,\hat{G}_n)\geq t \right) \\
\leq  & \Pb_{G^*} \left( \sup_{\myf\in \myF} \left|\frac{1}{n}\sum_{i\in [n]}t_\myf(X_i)- G^* \myf \right| \geq C(G_0)  t  \right), 
\end{align*}
and
\begin{align*}
& \Eb_{G^*} W_{2d_1-1}^{2d_1-1}(\hat{G}_n,G^*) 
\leq \\ & C(G_0)  \Eb_{G^*} \sup_{\myf\in \myF} \left|\frac{1}{n}\sum_{i\in [n]}t_\myf(X_i)- G^* \myf \right|. 
\end{align*}

\end{enumerate}
\end{thm}
In both claims of the above theorem, the convergence behavior of the estimate $\hat{G}_n$ toward the true mixing measure $G^*$, as quantified by the tail probability of a Wasserstein distance or its expectation, is upper bounded by that of the suprema of the empirical process associated with the function class 
$\{t_\phi |\phi \in \Phi\}$. The analysis of the latter is a well-studied subject of empirical process theory \citep{vanderVaart-Wellner-96,vandeGeer-00,wainwright2019high}; it is a standard exercise to derive their convergence rates for a specific class of functions.

\begin{rem}
\label{rem:minimax}
In the examples that follow we will see  that for a broad range of $\Phi$, it is possible to obtain a sharp bound for the tail probability in the right hand side of Eq.\eqref{eqn:globalconvergencerategeneralpro} and moreover,
\myredvsec{\begin{equation}
\label{eqn:empiricalprocessbound}
    \Eb_{G^*} \sup_{\phi\in \Phi} \biggr |\frac{1}{n}\sum_{i=1}^{n} t_\phi(X_i) - G^* \phi \biggr |
\lesssim n^{-1/2}
\end{equation}
where the constant hiding on the right hand side  does not depend on the unknown $G^*$. It follows from Theorem \ref{thm:convergencerate} part (a) that $\Eb_{G^*} W_{2k-1}(G^*,\hat{G}_n) \lesssim  n^{-1/(2(2k-1))}$, where the constant in this asymptotic expression does not depend on $G^*$.} This slow rate, which depends on the number of atoms $k$, is known to be the optimal rate of estimation for the mixing measure in the global minimax sense.
Part (b) of the theorem provides a refinement of the behavior of the estimator $\hat{G_n}$, if one localizes $G^*$ and $\hat{G}_n$ to a neighborhood of a probability measure $G_0$ that has only $k_0\leq k$ atoms. 
The rate, then, improves to $W_{2d_1-1}(G^*, \hat{G}_n)) = O_{p}(n^{-1/(2(2d_1-1))})$, where $d_1=k-k_0+1$, but the constant in $O_p$ is dependent of $G_0$. This rate is known to be the optimal rate of estimation for the mixing measure in the local minimax sense.
\end{rem}

Suppose that one is only interested in characterizing a pointwise convergence behavior of the estimate $\hat{G}_n$ to $G^*$, i.e., by allowing the constant in the asymptotic expression to be dependent on the unknown $G^*$, then one should expect a faster rate of convergence. In this case, it suffices to appeal to the pointwise inverse bounds such as \eqref{eqn:localinverseboundfixargument} and \eqref{eqn:localinverseboundfixargumentexactfitted}. A counterpart of Theorem \ref{thm:convergencerate} is the following.

\begin{thm}
\label{thm:convergenceratepointwise}
Suppose that $\Theta$ is compact, $\Gc_k(\Theta)$ is distinguishable by $\myF$, and that $\myF$ is estimable on $\Gc_{k}(\Theta)$. 
Moreover, assume that $(\Gc_k,\Phi, G_0)$ possesses the pointwise inverse bound \eqref{eqn:localinverseboundfixargument} for any measure $G_0 \in \Gc_k(\Theta)$ such that $G_0$ has only $k_0$ atoms, for any $k_0 \leq k$. 
Let $\hat{G}_n$ be a minimum $\myF$-distance estimator.
Then for any $G^* \in \Gc_{k}(\Theta)$ there exists $C(G^*) > 0$ depending on $G^*$ such that for all sufficiently large $n$,
the following holds for all $t > 0$:
\begin{align*}
&\Pb_{G^*}\!\left(
  W_2^2(\hat G_n, G^*) \ge t
\right) \\
\; \le\; &
\Pb_{G^*}\!\left(
  \sup_{\phi \in \Phi}
  \left|
    \frac{1}{n}\sum_{i=1}^n t_\phi(X_i) - G^* \phi
  \right|
  \ge C(G^*) t
\right).
\end{align*}
and 
\[
\mathbb{E}_{G^*}
  W_2^2(\hat G_n, G^*)
\;\le\;
C(G^*) \,
\mathbb{E}_{G^*}
  \sup_{\phi \in \Phi}
  \left|
    \frac{1}{n}\sum_{i=1}^n t_\phi(X_i) - G^* \phi
  \right|.
\]
\end{thm}

\begin{comment}
As discussed in Remark \ref{rem:2.15},  a typical rate for the supreme of empirical processes is
    $\text{RHS} \asymp O_p\!\left(n^{-1/2}\right)$, so  
    $
    W_{2}\!\left(\hat G_n, G^\ast\right)
    \asymp
    O_p\!\left(n^{-\frac{1}{4}}\right)$. \myeoe
\end{comment}

As before, if \eqref{eqn:empiricalprocessbound} can be established, then we obtain the guarantee $W_{2}(\hat G_n, G^\ast) = O_p(n^{-\frac{1}{4}})$, where the constant in this asymptotic expression is dependent on $G^*$.

\begin{rem}
Theorem \ref{thm:convergenceratepointwise} shows that a direct estimation in the overfitted case might result in a slow convergence rate. As a counterpart, it is established that the two-step-procedure by firstly estimating the number of components and then estimating the mixing measure improves the rate to $n^{-\frac{1}{2}}$; see \cite{heinrich2018strong,wei2023minimum}. \myeoe
\end{rem}


\subsection{Examples}
\label{sec:MMD1}

Let us return to the examples of the test function class $\Phi$ described in Example \ref{exa:divationbetweenmixturedistributions}: \\
$\myF=\left\{ \theta \mapsto \int h(x) f(x|\theta) \dd \lambda(x) | h \in \Hc \right\}$ where $\Hc$ is a class of bounded and measurable functions on $(\Xf,\Xc)$. It was already noted earlier that this function class is estimable. Due to \eqref{eqn:integraldistance}, $\|G-H\|_\Phi$ may be viewed as $\Hc$-distance on the space of mixture distributions: $\|G-H\|_{\Phi} = \|P_G - P_H\|_{\Hc}$, which is also known as integral probability metric (IPM) distance. The minimum $\Phi$-distance estimator may be equivalently written as
\begin{align}
\label{eqn:ipm}
\hat{G}_n\in & \argmin_{G\in \Gc_k(\Theta)} \|\Pb_{G} - \hat{\Pb}_n  \|_{\Hc} ,
\end{align}
where $\hat{\Pb}_n:=\frac{1}{n}\sum_{i\in [n]}\delta_{X_i}$ denotes the empirical measures. Thus, we refer to the minimum $\myF$-distance estimators in this case as \emph{minimum IPM estimators} for which we shall present two different instances.

\vspace{.1in}
\noindent\underline{Minimum KS-distance estimator} 
%
Suppose that $(\Xf,\Xc)=(\Rb^d,\Bc(\Rb^d))$. Take $\Hc := \left\{1_{(-\infty,x]}(\cdot) \mid x\in \Rb^d \right\}$ where $(-\infty,x]:=(-\infty,x^{(1)}]\times \ldots \times(-\infty,x^{(d)}]$. 
For such $\Hc$, we have  $\myF = \myF_0 = \{ \theta \mapsto F(x \mid \theta) \mid x\in \Rb^d  \}$, where $F(x \mid \theta)$ stands for multivariate cumulative distribution function (cdf) of the probability kernel $\Pb_\theta(dx)$ on $\Rb^d$.
Accordingly, 
\begin{align*}
\|G -H \|_{\myF_0}= & \sup_{x \in \Rb^d} \left|\int_{(-\infty,x]}  d\Pb_G-\int_{(-\infty,x]} d\Pb_H\right| \\
=: & \KS(\Pb_G,\Pb_H), 
\end{align*}
where $\KS(\cdot,\cdot)$ stands for Kolmogorov–Smirnov (KS) distance. 
The corresponding minimum $\Phi$-distance (also minimum IPM) estimator
takes the following familiar form \citep{deely1968construction,chen1995optimal,heinrich2018strong}:
$$
\hat{G}_n \in \argmin_{G\in \Gc_k}\KS(\Pb_{G},\hat{\Pb}_n).
$$ 
That Eq.~\eqref{eqn:empiricalprocessbound} holds for the class $\Phi_0$ is a well-known fact in empirical process theory, i.e., a strengthening of Glivenko-Cantelli theorem via Dvoretzky-Kiefer-Wolfowitz inequality. Thus we have verified all the specific claims made in Remark \ref{rem:minimax}  for the corresponding minimum KS distance estimator. These claims were first established by \cite{heinrich2018strong} who corrected an earlier study of \cite{chen1995optimal} for the case $\Xf=\Rb$. For a precise statement of the uniform convergence theorem of the minimum KS distance estimator in the multivariate case of $\Xf$ as considered here, see \cite{wei2023minimum} (cf. Theorem 3.3). 

\vspace{.1in}
\noindent \underline{Minimum MMD estimator}
We can obtain a different estimator for the mixing measure by selecting another class of functions $\Hc$ in the minimum IPM estimation framework. In the remainder of this section we shall present one such instance. This instance is particularly interesting because it is applicable to a broader range of data domain $\Xf$, including when $\Xf$ is non-Euclidean and/or of high dimensions. 
Our approach is to take $\Hc$ to be a reproducing kernel Hilbert space of functions on $\Xf$, then the corresponding $\Phi$-distance is also known as the maximum mean discrepancy (MMD) in the RKHS literature \citep{gretton2012kernel}. Thus from here on,
we shall call this particular class of minimum IPM estimator the \emph{minimum MMD estimator}. 
As we shall see shortly, the choice of RKHS for function class $\Hc$ is uniquely desirable for several reasons: the function class is sufficiently "lean", so that guarantees regarding the associated empirical processes such as Eq.~\eqref{eqn:empiricalprocessbound} continue to hold. It is also sufficiently rich, so that identifiability conditions required to establish the inverse bounds assumed in the statement of Theorem \ref{thm:convergencerate} and \ref{thm:convergenceratepointwise} also hold as well. Finally, although not the focus of this article, the computation of the MMD is quite simple, making this estimator of potential interest from a practical viewpoint.

\begin{comment}
which we call the minimum MMD estimator. MMD stands for maximum mean discrepancy, a metric that arises from a particular choice of $\Phi$ using reproducing kernel Hilbert spaces (RKHS) \cite{gretton2012kernel}. 
We emphasize that while the minimum KS estimator may be difficult to apply to non-Euclidean or high-dimensional or complex structured data domain $\Xf$ (since the Kolmogorov-Smirnov distance evaluation involves finding the supremum over $\Xf$), the minimum MMD may be more applicable in such settings thanks to the powerful machinery of the RKHS. 
\end{comment}

 First, let us recall some basic background of the RKHS and the associated MMD metric. In this section, $\Xf$ is assumed to be any topological space endowed with the $\sigma$-algebra $\Xc=\Bc(\Xf)$, the Borel measurable sets.  
 Consider a real-valued symmetric and positive semidefinite kernel function $\ker(\cdot,\cdot)$ on the measurable space $(\Xf,\Xc)$.
Let $\Hc$ denote the  RKHS associated with the reproducing kernel $\ker(\cdot,\cdot)$  with its inner product $\langle \cdot, \cdot \rangle_\Hc$; 
for more details of RKHS, see, e.g., \cite{Saitoh88,steinwart2008support}. 

Denote by $\Mc_b(\Xf,\Xc)$ the space of all finite signed measures on $(\Xf,\Xc)$. Each $\Pb\in \Mc_b(\Xf,\Xc)$ defines  a linear map $h\mapsto Ph := \int_\Xf h d\Pb $ on $\Hc$. Suppose  $\ker(\cdot,\cdot)$ is bounded hereafter, i.e., $\|\ker\|_\infty:=\sup_{x\in \Xf} \sqrt{\ker(x,x)} <\infty $, and then the above linear map is bounded and hence $\Pb$ can be identified as a member $\mu(\Pb)$ in $\Hc$ by Riesz Representation Theorem \cite[Lemma 26]{sriperumbudur2011universality}, given as below:
\begin{equation*}
\label{eq:mu}
\mu(\Pb)(\cdot) = \int \ker(\cdot,x) d\Pb(x)\in \Hc, \quad \forall \Pb\in \Mc_b(\Xf,\Xc). 
\end{equation*}
\myredvsec{The map $\mu$ is also known as \emph{kernel mean embedding}.}

Denote by $\Pc(\Xf,\Xc)$ the space of all probability measures on $(\Xf,\Xc)$. 
Then, the \emph{maximum mean discrepancy} (MMD) associated with the kernel $\ker$ for a pair $\Pb,\Qb\in \Pc(\Xf,\Xc)$ is defined as, cf. \cite{gretton2012kernel}:
\begin{align*}
\numberthis \label{eq:mmd}
&\MMD^2(\Pb,\Qb) \\
:= & \Eb \ker(Z,Z') - 2 \Eb \ker(Z,Y) + \Eb \ker(Y,Y'),
\end{align*}
where $Z$ and $Z'$ are independent random variables with distribution $\Pb$, and $Y$ and $Y'$ independent random variables with distribution $\Qb$ \cite[Lemma 4, Lemma 6]{gretton2012kernel}. 

Next, a bounded measurable kernel is called a \emph{characteristic kernel} if the map $\mu: \Pc(\Xf,\Xc)\to\Hc$ is injective, i.e., $\MMD(\Pb,\Qb)=0$ if and only if $\Pb=\Qb\in \Pc(\Xf,\Xc)$. If a kernel is bounded, measurable and characteristic, then $\MMD$ is a valid a metric on $\Pc(\Xf,\Xc)$.  
Thus, the map $\mu$ provides a natural embedding of the space of probability measures $\Pc(\Xf,\Xc)$ into the RKHS $\Hc$ associated with the characteristic kernel $\ker(\cdot,\cdot)$. We refer to \cite{fukumizu2004dimensionality,muandet2017kernel,wei2023minimum} for further details on  MMD. 

Back to the IPM estimation framework \eqref{eqn:ipm}, let $\Hc_1$ be the unit ball of the RKHS $\Hc$, and take 
$$\myF = \myF_1= \biggr \{\phi: \theta \mapsto \phi(\theta)=\int h(x) f(\dd x|\theta)  \, \biggr | \, h\in \Hc_1  \biggr \}.$$ 
Then 
\begin{align}
    \|G -H \|_{\myF_1} = \|\Pb_G - \Pb_H \|_{\Hc_1}= \MMD(\Pb_G,\Pb_H),
\end{align}
the MMD between $\Pb_G$ and $\Pb_H$ as given in Eq. \eqref{eq:mmd}, where the second equality is a simple consequence of the reproducing property of kernel $\ker$.  
A minimum IPM estimator for this choice of $\myF_1$ is now called a minimum MMD estimator associated with the kernel function $\ker$, and takes the form
\begin{align*}
\hat{G}_n \in 
& \argmin_{G'\in \Gc_k(\Theta)} \MMD^2(\Pb_{G'},\hat{\Pb}_n) \\
= & \argmin_{\substack{p'\in \Delta^{k-1}\\ \theta'_1,\ldots,\theta'_k\in  \Theta}}  \sum_{i,j\in [k]} p'_i p'_j K(\theta'_i,\theta'_j) - 2 \sum_{i\in [k]} p'_i J_n(\theta'_i)
\end{align*}
where the last line follows from Eq.~\eqref{eq:mmd}, and
\begin{align*}
K(\theta,\theta')  =& \int \ker(z,z') f(\dd z| \theta) f(\dd z' | \theta'), \\
J_n (\theta) = &  \frac{1}{n} \sum_{i\in [n]} \int \ker(x,X_i) f(\dd x| \theta),
\end{align*}
with $\Delta^{k-1}:=\{p'\in \Rb^k \mid p'_i\geq 0, \ \sum_{i\in [k]}p'_i=1  \}$.  

\begin{comment}
\RestyleAlgo{ruled}
\begin{algorithm}
\caption{Minimum MMD estimators} \label{alg:three}
\KwData{$X_1,\ldots, X_n\overset{\iidtext}{\sim} \Pb_{G^*}$}
\KwResult{$\hat{G}_n$}

$(\hat{p},\hat{\theta}_1,\ldots,\hat{\theta}_k)\in \argmin_{\substack{p'\in \Delta^{k-1}\\ \theta'_1,\ldots,\theta'_k\in  \Theta}}  \sum_{i,j\in [k]} p'_i p'_j K(\theta'_i,\theta'_j) - 2 \sum_{i\in [k]} p'_i J_n(\theta'_i)$; \\ 
$\hat{G}_n =\sum_{i\in [k]} \hat{p}_i\delta_{\hat{\theta}_i}$
\end{algorithm}
\end{comment}

The minimum MMD estimators have been studied in \cite{briol2019statistical, cherief2022finite} for the purpose of density estimation. It was investigated for mixing measure estimation in \cite{wei2023minimum}, where the following theoretical guarantees for the convergence rates of the mixing measure were established by applying the general Theorem \ref{thm:convergencerate}.

\begin{thm}
\label{thm:MMDuniformrate}
Suppose that $\Theta$ is compact and the mixture model is identifiable on $\Gc_k(\Theta)$. Let $\hat{G}_n$ be a minimum MMD estimator. Consider a bounded and measurable kernel $\ker(\cdot,\cdot)$ on $(\Xf,\Xc)$ and  \myredvsec{assume that the kernel mean embedding} map $\mu: \Mc_b(\Xf,\Xc)\to \Hc$ is injective. 
 \begin{enumerate}[label=(\alph*)]
\item \label{itema:thm:MMDuniformrate}
Suppose the global inverse bound \eqref{eqn:globalinversebound} with $\myF=\myF_1$ holds. Then there exists a constant $C$, where its dependence on $ \Theta, k$ is suppressed, such that 
\begin{align*}
\sup_{G^*\in \Gc_k(\Theta)}	\Eb_{G^*} W_{2k-1}^{2k-1}(G^*,\hat{G}_n)
	\leq C \frac{\|\ker\|_\infty}{\sqrt{n}} . 
\end{align*}

\item \label{itemb:thm:MMDuniformrate}
Fix $G_0\in \Ec_{k_0}(\Theta)$. Suppose local inverse bound \eqref{eqn:localinversebound} with $\myF=\myF_1$ hold. 
Then for any $G_0\in \Ec_{k_0}(\Theta)$, there exists $r(G_0)$, $C(G_0)$ and $c(G_0)$, where their dependence on $ \Theta, k_0,k$ is suppressed, such that 

\begin{align*}
& \sup_{\substack{G^*\in \Gc_k(\Theta): \\W_1(G_0,G^*)<r(G_0) } } \Eb_{G^*} W_{2d_1-1}^{2d_1-1}(\hat{G}_n,G^*)  \\
\leq & C(G_0)  \frac{\|\ker\|_\infty}{\sqrt{n}}. 
\end{align*}

\end{enumerate}
\end{thm}
For brevity, we omit a counterpart statement regarding the pointwise convergence by applying Theorem \ref{thm:convergenceratepointwise}.

\section{Strong and weak identifiability theory}
\label{sec:identifiability}

 In Sections \ref{sec:intro} and \ref{sec:finite} we have demonstrated how inverse bounds play an essential role in  transferring convergence rates from the density level to the parameter level in finite mixture models. They are also instrumental in deriving new estimators. 
 The focus of this section is to establish such inverse bounds. We shall present sufficient conditions under which they can be derived; such conditions are broadly known as identifiability conditions. When sufficiently strong identifiability conditions are satisfied for the kernel $f$, best possible inverse bounds may be established under considerable generality. When strong identifiability conditions are violated, we will walk into the territory of a weak identifiability theory, which is discussed in Sections \ref{sec:weakidentifiability} and \ref{sec:singular}. 

\subsection{Sufficient conditions for uniform inverse bounds}
\label{sec:stronguniform}


%

In this subsection we provide sufficient conditions to establish uniform inverse bounds \eqref{eqn:globalinversebound} and \eqref{eqn:localinversebound}.
Because inverse bounds represent a refinement of identifiability condition, we should expect the sufficient conditions to be stronger than the standard identifibility condition. Due to the fact that the mixture distribution $P_G$ is but a linear combination of the kernel components $\{f(\cdot|\theta)|\theta \in \Theta\}$, the standard identifiability condition for finite mixture models is typically stated in terms of linear independence of these kernel bases. A natural way of strengthening the standard identifiability condition is to consider a collection of functions obtained by taking partial derivatives up to a certain order of the kernel $f$ with respect to all parameters of interest, and to insist that such functions be linearly independent. 

This leads to the following standard notion of strong identifiability which is defined more generally on the collection of test functions $\Phi$. For a multi-index $\alpha$, the operator $D^{\alpha}$ means partial derivative of order $\alpha^{(i)}$ to the $i$-th coordinate. Note in this article that the partial derivative is always with respect to $\theta$, i.e. $D^\alpha f(x \mid \theta)=\frac{\partial^\alpha }{\partial \theta^{\alpha}} f(x \mid \theta)$.

\begin{definition}[$m$-strong identifiability] 
\label{def:mstrong}
A family $\myF$ of functions of $\theta$ is $m$-strongly identifiable if each $\myf\in\myF$ is $m$-order continuously differentiable, and for any finite set of $\ell$ distinct points $\theta_i\in \Theta$, 
$$
\sum_{i=1}^{\ell}\  \sum_{|\alpha|\leq m} a_{i\alpha} D^\alpha \myf(\theta_i)  = 0,  \quad \forall \myf\in \myF
$$    
if and only if 
$$
a_{i\alpha}=0, \quad \forall \ 0\leq |\alpha|\leq m, \ i\in [\ell].
$$
\end{definition}

If we take
  \(
    \Phi = \bigl\{ \theta \mapsto f(x \mid \theta) \;\big|\; x \in \mathcal{X} \bigr\},
  \)
  where $f$ is the mixture’s probability kernel (density), then we are reduced
  to the classical \emph{strong identifiability} condition. This condition holds
  for many families, for example exponential families with a single parameter. This classical notion of strong identifiability has been widely studied in the literature \cite{chen1995optimal,nguyen2013convergence,ho2016strong,heinrich2018strong,ho2019singularity,wei2023minimum}, which are roughly the linear independence between the mixture kernel density (or cdf) and its derivatives.

 Still, for certain class of test functions such as the monomial family,
  \(
    \Phi_2 = \bigl\{ (\theta - \theta_0)^{\alpha} \bigr\}_{|\alpha| \le 2k-1},
  \)
  in Example~\ref{exa:moment}, $m$-strong identifiability condition as defined may be violated. This motivates the following weaker sufficient conditions. 

\begin{definition}
\label{def:linearindependentdomain}
The family $\myF$ is said to be a \textit{$(m,k_0,k)$ linear independent domain} if the following hold: 1) Each $\myf\in \myF$ is $m$-th order continuously differentiable on $\Theta$; 
and 2) Consider any integer $\ell\in [k_0,2k-k_0]$, and any vector $(m_1,m_2,\ldots,m_\ell)$ such that $1\leq m_i\leq m+1$ for $i\in [\ell]$ and $\sum_{i=1}^\ell m_i\in [2k_0, 2k]$, then
for any distinct $\{\theta_i\}_{i\in [\ell] }\subset \Theta$, the  operators $\{D^\alpha|_{\theta=\theta_i}\}_{ 0\leq |\alpha|< m_i, i\in [\ell]}$  on $\myF$ are linearly independent, i.e.,  
\begin{subequations} \label{eqn:linearinddomain}
\begin{align}
\sum_{i=1}^{\ell}\  \sum_{|\alpha|\leq m_i-1} a_{i\alpha} D^\alpha \myf(\theta_i)  = &0,  \quad \forall \myf\in \myF \label{eqn:linearinddomaina}\\
\sum_{i\in [\ell]}   a_{i\bm{0}}   = &0, \label{eqn:linearinddomainb}
\end{align}
\end{subequations}
if and only if 
$$
a_{i\alpha}=0, \quad \forall \ 0\leq |\alpha|< m_i, \ i\in [\ell].
$$
\end{definition}
\begin{rem} \label{rem:linearinddomainconstraint}
The equation \eqref{eqn:linearinddomainb} can be seen as \eqref{eqn:linearinddomaina} with $\myf \equiv 1_\Theta$, the constant function $1$ on $\Theta$. So a slightly more accurate terminology should be ``the  operators $\{D^\alpha|_{\theta=\theta_i}\}_{ 0\leq |\alpha|< m_i, i\in [\ell]}$  on $\myF\cup \{1_\Theta\}$ are linear independent''. 
%
%
It is easy to see that if $\myF$ is a $(m,k_0,k)$ linear independent domain, then so is any of its supersets $\myF$ (assuming they contain only sufficiently differentiable functions so that \eqref{eqn:linearinddomaina} and \eqref{eqn:linearinddomainb} may be verified). Another useful observation is that if $\myF$ is a $(m,k_0,k)$ linear independent domain then $\myF$ is a $(m',k',k)$ linear independent domain for any $m'\leq m$ and $k'\geq k_0$.
\myeoe
\end{rem}
The advantage of the definition of $m$-strong identifiability is that it is simpler and  does not depend on the constraints imposed by the number of atoms $k$ and $k_0$. However, the former condition is clearly much more stringent: $\myF$ is $m$-strongly identifiable implies that $\myF$ is a $(m,k_0,k)$ linear independent domain for any $k_0\leq k$. The $(m,k_0,k)$ linear independent domain condition is a relaxation of $m$-strong identifiability thanks to the reduced number of equations required in Eq. \eqref{eqn:linearinddomaina} and \eqref{eqn:linearinddomainb}, ones that arise from a careful consideration of possible allocations of atoms of $k$-component mixing measures converging to a fixed $k_0$-component mixing measure. This relaxation can be impactful, especially when the function class $\Phi$ is finite.  
In that scenario, the linear system in the definition of linear independent domain is much better behaved than that in the definition of strong identifiability, since the former has less variables while the number of equations remain the same. In particular, there are several notable examples, e.g., family of monomials $\myF_2$ (see Example~\ref{exa:moment} and Section \ref{sec:momentdeviation}), that satisfy the $(m,k_0,k)$ linear independent domain condition but not $m$-strongly identifiability condition. In addition, it is shown in \cite[Example~7.6]{wei2023minimum} for mixtures of multinomial distributions, by using the weaker condition of linear independent domain, \myredvsec{the inverse bounds hold if and only if the number of trials $N\geq 2k-1$, where $k$ is the number of mixture components. This improves the previous results \cite[Proposition 1 and Corollary 1]{manole2021estimating}. } 

The following general theorem establishes that $(m,k_0,k)$ linear independent domains are sufficient conditions for establishing the uniform local inverse bounds.


\begin{thm}
\label{thm:inversebound}
Suppose that $\Theta\subset \Rb^q$ is compact. Let $d_1=k-k_0+1$.
\begin{enumerate}[label=(\alph*)]

\item \label{thm:inversebounda}
If $\myF$ is a $(2d_1-1,k_0,k)$ linear independent domain, then \eqref{eqn:localinversebound} holds for any $G_0\in \Gc_k(\Theta)$ such
that $G_0$ has exactly $k_0$ atoms in $\Theta$. 
\item \label{thm:inverseboundb}
If $\myF$ is a $(2k-1,1,k)$ linear independent domain, then \eqref{eqn:localinversebound} holds for any $G_0 \in \Gc_k(\Theta)$ that has exactly $k_0$ atoms in $\Theta$, for any $k_0 \leq k$. 
\end{enumerate}
\end{thm}
We note that the compactness assumption is an important condition for the theorem's claims to hold. More discussion on this will be given in the sequel.
By Remark \ref{rem:linearinddomainconstraint} if $\myF$ is a $(2k-1,1,k)$ linear independent domain this implies that $\myF$ is a $(2d_1-1,k_0,k)$ linear independent domain for any $k_0\in [k]$, hence in Theorem \ref{thm:inversebound} part \ref{thm:inverseboundb} immediately follows from part \ref{thm:inversebounda}.
It is also noted that
in \cite[Section~8.5]{wei2023minimum} that the exponent $2d_1-1$ of the denominator in \eqref{eqn:localinversebound} is optimal. 
Finally, the uniform global inverse bound \eqref{eqn:globalinversebound} also immediately follows in lieu of  Lemma \ref{lem:localtoglobal}. 
\begin{rem}
Notice that $\sup_{\myf\in \myF} |G \myf-H \myf| = \sup_{\myf\in \myF\cup \{1_\Theta\}} |G \myf-H \myf| $ since $G 1_\Theta - H 1_\Theta = 0$. So we may always assume that $1_\Theta\in \myF$ without affecting \eqref{eqn:localinversebound}. For $\myf=1_\Theta$, the corresponding $t_\myf=1_\Xf$. Hence the minimum $\myF$-distance estimator $\hat{G}_n$ also remains unchanged by replacing $\myF$ with $\myF\cup \{1_\Theta\}$. See Remark \ref{rem:linearinddomainconstraint} for a related discussion.
\myeoe
\end{rem}


The special case that $\myF=\myF_0$ (cf. KS distance in Example \ref{exa:divationbetweenmixturedistributions}) of Theorem \ref{thm:inversebound} under $m$-strong identifiability was first obtained in \cite[Theorem 6.3]{heinrich2018strong}; while the general version of inverse bounds as presented above was given by \cite[Theorem 2.21]{wei2023minimum}.


\subsection{Example: moment based estimators}
\label{sec:momentdeviation}

In this section we consider the monomial family $\myF_2:=\{(\theta-\theta_0)^\alpha\}_{\alpha\in \Ic_{2k-1}}$ in Example~\ref{exa:moment}, where $\theta_0$ is an arbitrarily chosen element in $\Rb^q$. 
Since discrete distributions with $k$ support points are uniquely characterized by their first $2k-1$ moments, but not their first $2k-2$ moments by \cite[Lemma~8.1]{wei2023minimum}, 
we know that $\Gc_k(\Theta)$ is distinguishable by $\myF_2$. 

\vspace{.1in}
\noindent \underline{Inverse bounds}
\begin{comment}
We first show that $\myF_2$ is a $(2k-1,1,k)$ linear independent domain. It is obvious that each monomial in $\myF_2$ is $2k-1$ differentiable. Consider any integer $\ell\in [1,2k-1]$, and any vector $(m_1,m_2,\ldots,m_\ell)$ such that $1\leq m_i\leq 
2k$ for $i\in [\ell]$ and $\sum_{i=1}^\ell m_i\in [2, 2k]$. Consider any distinct $\{\theta_i\}_{i\in [\ell] }\subset \Theta$. The equations \eqref{eqn:linearinddomaina} \eqref{eqn:linearinddomainb} become
\begin{align}
\sum_{i=1}^{\ell}\  \sum_{|\gamma|\leq m_i-1} a_{i\gamma} \frac{\alpha!}{(\alpha-\gamma)!}  (\theta_i-\theta_0)^{\alpha-\gamma}1_{\alpha\geq \gamma}  = &0,  \quad \alpha \in \Ic_{2k-1}. 
\end{align}
It then follows from \cite[Lemma~8.8]{wei2023minimum} 
that $a_{i\gamma}=0$ for any $\gamma\in \Ic_{m_i-1}, i\in [\ell]$. So $\myF_2$ is a $(2k-1,1,k)$ linear independent domain. (Note that it is straightforward to see that $\myF_2$ is not $(2k-1)$-strongly identifiable.) 
Provided that $\Theta$ is compact, then
%
we may apply Theorem \ref{thm:inversebound}, which yields that 
\end{comment}
%
It is not difficult to verify that $\myF_2$ is not $(2k-1)$-strongly identifiable. 
The next lemma shows that $\myF_2$ is a $(2k-1,1,k)$ linear independent domain.

%

\begin{comment}
$\myF_2:=\{(\theta-\theta_0)^\alpha\}_{\alpha\in \Ic_{2k-1}}$ but not $\myF_1$ unless $d_1=k$ or equivalently $k_0=1$. Therefore we will work on the larger family $\myF_2$. Note that since $\myF_2\supset \myF_2$, \eqref{eqn:localinversebound} holds for $\myF=\myF_2$,  for any $G_0\in \Gc_k(\Theta)$, and is as below: 
\begin{equation*}
	\liminf_{ \substack{G,H\overset{W_1}{\to} G_0\\ G\neq H\in \Gc_k(\Theta) }} \frac{\|\mbf_{2k-1}(G-\theta_0)-\mbf_{2k-1}(H-\theta_0)\|_{\infty}}{W_{2d_1-1}^{2d_1-1}(G,H)} >0. 
\end{equation*}
\end{comment}

\begin{lem}
\label{lem:momentbound}
The family $\myF_2$ is a $(2k-1,1,k)$ linear independent domain, and $\Gc_k(\Theta)$ is distinguishable by $\myF_2$. 
\end{lem}
If additionally $\Theta\subset \Rb^q$ is  compact, then by Theorem \ref{thm:inversebound} and Lemma \ref{lem:localtoglobal}, we obtain both local and global inverse bound. 
The local inverse bound
\eqref{eqn:localinversebound} with $\myF=\myF_2$ takes the form: 
\begin{equation}
\liminf_{ \substack{G,H\overset{W_1}{\to} G_0\\ G\neq H\in \Gc_k(\Theta) }} \frac{\|\mbf_{2k-1}(G-\theta_0)-\mbf_{2k-1}(H-\theta_0)\|_{\infty}}{W_{2d_1-1}^{2d_1-1}(G,H)} >0. \label{eqn:localmomentbound}
\end{equation}
The global inverse bound
\eqref{eqn:globalinversebound} with $\myF=\myF_2$ 
takes the form:
\begin{equation}
	\inf_{  G\neq H\in \Gc_k(\Theta) } \frac{\|\mbf_{2k-1}(G-\theta_0)-\mbf_{2k-1}(H-\theta_0)\|_{\infty}}{W_{2k-1}^{2k-1}(G,H)} >0. \label{eqn:globalmomentbound}
\end{equation}

The univariate case $q=1$ of \eqref{eqn:globalmomentbound} was first established by \cite[Proposition 1]{wu2020optimal}. 
 \myred{ The \eqref{eqn:globalmomentbound} for general $q$ was implied by the Theorem 4.2 and Equation (4.49) in \cite{doss2020optimal}.  It is worth mentioning that both previous bounds specify the dependence on the parameters $k$ and $q$. 
 Lemma \ref{lem:momentbound} is from  \cite[Lemma 3.16]{wei2023minimum}, which is obtained by specializing Theorem \ref{thm:inversebound} 
 for $\myF=\myF_2$ and the linear independence of multinomial (\cite[Lemma~8.8]{wei2023minimum}). 
 }

\vspace{.1in}
\noindent \underline{Minimum $\Phi_2$-distance estimators}
So far the discussion only concerns properties about discrete distributions in $\Gc_k(\Theta)$ and does not involve mixture models or the probability kernel $\{f(\cdot|\theta)\}_{\theta\in\Theta}$. As described in the previous section, to ensure that $\myF_2$ is estimable, it is required that for each $\myf = \left(\theta-\theta_0\right)^\alpha\in \myF_2$, where $\alpha\in \Ic_{2k-1}$, there exists a function $t_\alpha$ defined on $\Xf$ such that 
\begin{equation}
G \myf= m_\alpha(G-\theta_0) =\Eb_{G}t_\alpha(X), \quad \forall G\in \Gc_k(\Theta). \label{eqn:momenttj}
\end{equation}  
A minimum $\myF$-distance estimator in this case becomes 
\begin{equation}
\hat{G}_n\in \argmin_{G'\in \Gc_k(\Theta)} \sup_{\alpha\in \Ic_{2k-1}}\left|m_\alpha(G'-\theta_0) -\frac{1}{n}\sum_{i\in [n]}t_\alpha(X_i)\right|.  \label{eqn:momentmethod}
\end{equation}
This is the \emph{generalized method of moments} (GMM).
%
The vector $\left(\frac{1}{n}\sum_{i\in [n]}t_\alpha(X_i)\right)_{\alpha\in \Ic_{2k-1}}$
is an empirical estimate of $\mbf_{2k-1}(G-\theta_0)$, and might not lie in a valid moment space for a discrete distribution due to the randomness, but the parameter estimate may be obtained by finding the closest corresponding moment vector w.r.t. $\|\cdot\|_\infty$. Specializing 
the GMM estimator in \eqref{eqn:momentmethod} when the probability kernel $f(\cdot|\theta)$ is a univariate Gaussian distribution, we recover the "denoised method of moments" algorithm that was investigated by \cite{wu2020optimal}. The generalized method of moments as given by \eqref{eqn:momentmethod} can be applied to other types of kernels $f$ and was formulated by \cite{wei2023minimum}.  


\begin{comment}
\begin{algorithm}
\caption{Generalized method of moments} \label{alg:moments}
\KwData{$X_1,\ldots, X_n\overset{\iidtext}{\sim} \Pb_{G^*}$}
\textbf{Parameter}: $\theta_0$ \\
\KwResult{$\hat{G}_n$}
$\bar{t}_\alpha(\theta_0) \gets \frac{1}{n}\sum_{i\in [n]}t_\alpha(X_i)$, for $\alpha\in \Ic_{2k-1}$\;
$\hat{G}_n\in \argmin_{G'\in \Gc_k(\Theta)} \| \mbf_{2k-1}(G'-\theta_0) - \bar{\tbf}  \|_\infty$, where $\bar{\tbf}=(\bar{t}_\alpha(\theta_0))_{\alpha\in \Ic_{2k-1}}$.
\end{algorithm}
\end{comment}



We now state Theorem \ref{thm:convergencerate} specialized to the GMM estimators, combined with Lemma \ref{lem:momentbound}. 

\begin{thm}
\label{thm:convergencerategeneralmoment}
	Suppose that $\Theta$ is compact. Suppose that for each $\alpha \in \Ic_{2k-1}$, there exists a real-valued function $t_\alpha$ defined on $\Xf$ such that \eqref{eqn:momenttj} holds. Let $\hat{G}_n$ be the GMM estimator from \eqref{eqn:momentmethod}.
 \begin{enumerate}[label=(\alph*)] 
\item \label{itema:thm:convergencerategeneralmoment}
Then there exists  $C$, where its dependence on $ \Theta, k$ and the probability kernel $f$ is suppressed, such that for any $G^*\in \Gc_k(\Theta)$, 
\begin{comment}
\begin{align*} 
\numberthis \label{eqn:convergenceratemomentgeneral}
	& \Pb_{G^*} \left(W_{2k-1}^{2k-1}(G^*,\hat{G}_n)\geq t \right) \\
	\leq &  \Pb_{G^*} \left( \sup_{\alpha\in \Ic_{2k-1}}\left| \frac{1}{n}\sum_{i\in [n]}t_\alpha(X_i)-m_\alpha(G^*-\theta_0)\right| \geq C  t  \right). 
\end{align*}
and
\end{comment}
\begin{align*}
	& \Eb_{G^*} W_{2k-1}^{2k-1}(G^*,\hat{G}_n) \\
	\leq & C \Eb_{G^*}  \sup_{\alpha\in \Ic_{2k-1}}\left| \frac{1}{n}\sum_{i\in [n]}t_\alpha(X_i)-m_\alpha({G^*}-\theta_0)\right| . 
\end{align*}
 
\item \label{itemb:thm:convergencerategeneralmoment}
Fix $G_0\in \Gc_k(\Theta)$ such that $G_0$ has exactly $k_0$ atoms.  Then there exists $r(G_0)$, $C(G_0)$ and $c(G_0)$, where their dependence on $ \Theta, k_0,k$ and the probability kernel $f$ is suppressed, such that for any ${G^*} \in \Gc_k(\Theta)$ satisfying $W_1(G_0,{G^*})<r(G_0)$,
\begin{comment}
\begin{align*} 
\numberthis \label{eqn:convergenceratemomentlocalgeneral}
& \Pb_{G^*} \left(D({G^*},\hat{G}_n)\geq t \right) \\
\leq  & \Pb_{G^*} \left( \sup_{\alpha\in \Ic_{2k-1}}\left| \frac{1}{n}\sum_{i\in [n]}t_\alpha(X_i)-m_\alpha({G^*}-\theta_0)\right| \geq C(G_0)  t  \right)
\end{align*}
and
\end{comment}
\begin{align*}
& \Eb_{G^*} W_{2d_1-1}^{2d_1-1}(\hat{G}_n,{G^*})  \leq \\
&  C(G_0)  \Eb_{G^*} \sup_{\alpha\in \Ic_{2k-1}}\left| \frac{1}{n}\sum_{i\in [n]}t_\alpha(X_i)-m_\alpha({G^*}-\theta_0)\right|. 
\end{align*}
\end{enumerate}
	\end{thm}
The main assumption in Theorem~\ref{thm:convergencerategeneralmoment} is that $\myF_2$ is estimable. For the univariate case $d=q=1$, it is shown that, when the probability kernel $\{f(\cdot|\theta)\}$ belongs to the \emph{the natural exponential families with quadratic variance functions} (NEF-QVF) \cite{morris1982natural,morris1983natural}, $\myF_2$ is estimable \cite{lindsay1989moment,wei2023minimum}. For GMM estimators in Gaussian mixture models, the convergence rate in the univariate case is established in 
\cite{wu2020optimal} , while the multivariate case is obtained in \cite{pereira2022tensor,doss2020optimal,wei2023minimum}. 

\subsection{Sufficient conditions for pointwise inverse bounds}
\label{sec:strongpointwise}
Next, we turn to sufficient conditions to establish pointwise inverse bounds \eqref{eqn:localinverseboundfixargument} and \eqref{eqn:localinverseboundfixargumentexactfitted}.  As discussed in Section \ref{sec:finite}, such inverse bounds are useful for establishing pointwise rates of convergence for a variety of estimators. Such pointwise inverse bounds can be established under a suitable strong identifiability conditions that are considerably less stringent than those required for establishing the uniform inverse bounds.

\begin{definition} \label{def:2ndfixoneargument}
Fix $G_0= \sum_{i=1}^{k_0}p_i^0\delta_{\theta_i^0}$, a discrete probability measure that has exactly $k_0 \leq k$ atoms in $\Theta$.
A family $\myF$ is said to be a \textit{$(G_0,k)$ second-order linear independent domain} if the following hold: 1) Each $\myf\in \myF$ is second-order continuously differentiable at $\theta_i^0$ for each $i\in [k_0]$; and  2) Consider any integer $\ell_1\in [k_0]$, and $\ell\in [k_0,k]$. Set $m_i=2$ for $i\in [\ell_1]$, $m_i=1$ for $\ell_1<i\leq k_0$ and $m_i=0$ for $k_0<i\leq \ell$. For any distinct $\{\theta_i^0\}_{i=k_0+1}^{\ell}\subset \Theta \setminus \{\theta_i^0\}_{i\in [k_0]}$, the  operators $\{D^\alpha|_{\theta=\theta_i^0}\}_{ 0\leq |\alpha|\leq m_i, i\in [\ell]}$  on $\myF$ are linearly independent, i.e.,  
\begin{comment}
\begin{subequations}
\begin{align}
\sum_{i=1}^{\ell_1}\  \sum_{|\alpha|\leq 2} a_{i\alpha} D^\alpha \myf(\theta_i) + \sum_{i=\ell_1+1}^{k_0}\  \sum_{|\alpha|\leq 1} a_{i\alpha} D^\alpha \myf(\theta_i) +  \sum_{i=k_0+1}^{\ell}\ a_{i\bm{0}}  \myf(\theta_i) = &0,  \quad \forall \myf\in \myF \label{eqn:linearinddomainafixed}\\
\sum_{i\in [\ell]}   a_{i\bm{0}}   = &0, \label{eqn:linearinddomainbfixed}
\end{align}
\end{subequations}
\end{comment}
\begin{subequations}
\begin{align}
\sum_{i=1}^{\ell}\  \sum_{|\alpha|\leq m_i} a_{i\alpha} D^\alpha \myf(\theta_i^0)  = &0,  \quad \forall \myf\in \myF \label{eqn:linearinddomainafixed}\\
\sum_{i\in [\ell]}   a_{i\bm{0}}   = &0, \label{eqn:linearinddomainbfixed}
\end{align}
\end{subequations}
if and only if 
$$
a_{i\alpha}=0, \quad \forall \ 0\leq |\alpha|\leq m_i, \ i\in [\ell].
$$
\end{definition}

It is clear that $\myF$ is $m$-strongly identifiable for $m=2$ implies that $\myF$ is a $(G_0,k)$ second-order linear independent domain for any $k\geq 1$ and $G_0\in \Ec_{k_0}(\Theta)$ with $k_0\in [k]$. The following is the pointwise counterpart to the uniform version in Theorem \ref{thm:inversebound}:

\begin{comment}
\begin{rem}
In principle, one can also have a stronger inverse bound (and upper bound) in terms of moment difference when $G_0$ is fixed, in the spirit of Theorem \ref{thm:inverseboundmoment}. The proof should be similar and we leave the details to interested readers. \myeoe
\end{rem}
\end{comment}

\begin{thm}
\label{lem:inverseboundfixoneargument}
Fix $G_0 \in \Gc_{k}(\Theta)$ such that $G_0$ has exactly $k_0 \leq k$ atoms. Suppose that $\myF$ is a $(G_0,k)$ second-order linear independent domain and that $\Theta\subset \Rb^q$ is compact. Then the pointwise inverse bound \eqref{eqn:localinverseboundfixargument} holds.
\end{thm}

The pointwise inverse bound \eqref{eqn:localinverseboundfixargument} with specific test function classes $\Phi$ that arise from the kernel $f$ was obtained in \cite{chen1995optimal,nguyen2013convergence,ho2016strong}. Theorem \ref{lem:inverseboundfixoneargument} is from \cite{wei2023minimum}. 

Next, we present sufficient conditions for the exact-fitted pointwise inverse bound \eqref{eqn:localinverseboundfixargumentexactfitted}. 

\begin{definition} \label{def:1stfixoneargument}
A family $\myF$ is said to be a \textit{$(G_0,k_0)$ first-order linear independent domain} for $G_0=\sum_{i=1}^{k_0}p_i^0\delta_{\theta_i^0}\\ 
\in \Ec_{k_0}(\Theta)$ if the following hold. 1) Each $\myf\in \myF$ is first-order continuously differentiable at $\theta_i^0$ for each $i\in [k_0]$.  2) The  operators $\{D^\alpha|_{\theta=\theta_i^0}\}_{ 0\leq |\alpha|\leq 1, i\in [k_0]}$  on $\myF$ are linearly independent, i.e.,  
\begin{subequations}
\begin{align}
\sum_{i=1}^{k_0}\  \sum_{|\alpha|\leq 1} a_{i\alpha} D^\alpha \myf(\theta_i^0)  = &0,  \quad \forall \myf\in \myF \label{eqn:linearinddomainafixed1st}\\
\sum_{i\in [k_0]}   a_{i\bm{0}}   = &0, \label{eqn:linearinddomainbfixed1st}
\end{align}
\end{subequations}
if and only if 
$$
a_{i\alpha}=0, \quad \forall \ 0\leq |\alpha|\leq 1, \ i\in [k_0].
$$
\end{definition}

It is clear that if $\myF$ is a $(2d_1-1,k_0,k)$ linear independent domain then $\myF$ is a $(G_0,k_0)$ first-order linear independent domain for any $G_0\in \Ec_{k_0}(\Theta)$. It also follows that $\myF$ is $m$-strongly identifiable for $m=1$ implies that $\myF$ is a $(G_0,k_0)$ first-order linear independent domain for any $G_0\in \Ec_{k_0}(\Theta)$ for any $k_0\geq 1$. Next is a result on the exactfitted pointwise inverse bound. 

\begin{lem}
\label{lem:inverseboundfixoneargumentexactfitted}
Fix $G_0=\sum_{i=1}^{k_0}p_i^0\delta_{\theta_i^0}$ which has exactly $k_0$ atoms. Suppose that $\myF$ is a $(G_0,k_0)$ first-order linear independent domain. Then \eqref{eqn:localinverseboundfixargumentexactfitted} holds.
\end{lem}

   The exact-fitted pointwise inverse bound \eqref{eqn:localinverseboundfixargumentexactfitted} with specific $\myF$ is obtained in \cite{ho2016strong,wei2022convergence}, while the current presented version with the general $\Phi$-distance in Lemma \ref{lem:inverseboundfixoneargumentexactfitted} is  from \cite[Lemma 4.7]{wei2023minimum}.

Finally, we make several remarks regarding the optimality of the inverse bounds which were obtained under the formulated strong identifiability conditions. As we have seen in Theorem \ref{thm:inversebound} and \ref{lem:inverseboundfixoneargument}, the strong identifiability conditions are expressed in terms of the test function class $\Phi$, as well as the number of atoms for the finite measure $G_0$, relatively to the dimensionality of its ambient space $\Gc_k(\Theta)$. Moreover, aside from the case $k=k_0$, $\Theta$ was assumed to be compact. In fact, the compactness of $\Theta$ was an essential ingredient in the proof-by-contradiction arguments employed in these theorems. 

It may be verified that the order $r$ in $W_r^r(G,G')$ which appeared in the inverse bounds \eqref{eqn:globalinversebound} \eqref{eqn:localinversebound} \eqref{eqn:localinverseboundfixargument} and \eqref{eqn:localinverseboundfixargumentexactfitted}, namely, $r=2k-1, r=2d_1-1, r=2$
and $r=1$, respectively, is the best one can establish under such suitably strong identifiability conditions. 
The first two cases can be found for the monomial class in \cite{wei2023minimum}  (Theorem 2.24); the last two cases may be justified via Theorem 3.2(b) and the remark following Theorem 3.1 of \cite{ho2016strong}.

   The boundedness of $\Theta$ is necessary in \eqref{eqn:globalinversebound} \eqref{eqn:localinversebound} and \eqref{eqn:localinversebound}, if one is to retain the aforementioned order $r$ in $W_r^r$.   This can be seen by the following lemma \cite[Lemma 4.9]{wei2023minimum}:

\begin{lem}
\label{lem:compactness}
Suppose that $\Theta=\Rb^q$ and the function class $\myF$ is uniformly bounded, i.e.  $\sup_{\myf\in \myF} \sup_{\theta\in \Theta}|\myf(\theta)|<\infty$.  Fix $G_0 \in \Ec_{k_0}(\Theta)$ such that $G_0$ has exactly $k_0 \leq k$ atoms. 
Then for any $r>0$, 
\begin{equation}
\liminf_{ \substack{G\overset{W_1}{\to} G_0\\ G\in \Gc_k(\Theta) }} \frac{\|G-G_0 \|_{\myF}}{W_{r}^{r}(G,G_0)} =0. 
\end{equation}
\end{lem}

In general, if $\Theta$ is unbounded, then weaker inverse bounds, such as the ones given in Theorem \ref{thm:convolution} for convolution mixtures may be available, provided that a moment constraint is imposed on the subset of mixing measures in $\Pcal(\Theta)$ (cf. Theorem 2 and the following remark in \cite{nguyen2013convergence}).

\subsection{Weak identifiability theory}
\label{sec:weakidentifiability}

In this section we focus on the test functions which arise directly from the probability kernel $f$, namely 
\begin{equation}
\label{eqn:kernelclass}
    \Phi = \{\theta \mapsto f(x|\theta)|x \in \Xf\}.
\end{equation} 
As mentioned earlier, much of the existing theory for mixture models is related to the analysis of this function class due to the ease of verification using the classical $m$-strong identifiability condition. 

Specifically, if we restrict the kernel $f$ to a translation-invariant kernel,  i.e.
\[
f(x \mid \theta) = g(x - \theta)
\qquad \text{for some } g : \mathbb{R} \to \mathbb{R}.
\]
so that the resulting mixture is a convolution mixture, then the $m$-strong identifiability condition holds quite easily, due to the following result (cf. Theorem 3 of \cite{chen1995optimal} and Lemma~B.2 of \cite{wei2023minimum}):

\begin{lem}
\label{eqn:locationmixture}
Let $g(x)$ be a function on $\Rb$ that is $m$-th order differentiable for every $x\in \Rb$ and that the $j$-th derivative $\frac{d^j}{dx^j}g(x)$ is Lebesgue integrable for any $j\in [m]$. Then the convolution mixture with density (w.r.t. Lebesgue measure) kernel $f(x \mid \theta)=g(x-\theta)$ is $m$-strongly identifiable. 
\end{lem}


Unfortunately, as soon as one moves beyond translation-invariant kernels, e.g., when the parameter $\theta$ in the probability kernel $f(x|\theta)$ is composed of parameters of different types (such as location, scale, and/or skewness) then $m$-strong identifiability conditions are typically violated, even for $m=1,2$ and so on. 
This is especially common for kernels which carry some physical interpretation and that may be originally motivated as solutions of some partial differential equations. In the examples that follow, such PDEs correspond to the violation of linear independence for the partial derivatives of the kernel $f$. This turns out to have a deep impact on the convergence behavior of the mixing measure $G$ for most standard estimators.


\begin{exa}[Location-scale Gaussian kernel] \label{exa:weakgau}
Consider the Gaussian kernel on $\R^d$
    \[
        f(x \mid \theta, \Sigma)
        = \frac{1}{\sqrt{(2\pi)^{d} |\Sigma|}}
        \exp\left(-\frac{1}{2}\langle x-\mu, \Sigma^{-1}(x-\mu)  \rangle \right)
    \]
One can check that
\begin{equation}
\label{eqn:pde-gauss}
    \frac{\partial^2 f}{\partial \theta^2}
    =
    2\,\frac{\partial f}{\partial \Sigma}.
\end{equation}
This is known as the \emph{heat equation}.  
This means that $f(x \mid \theta, \Sigma)$ fails to satisfy second-order identifiability, due to the interaction between the mean parameter $\theta$ and covariance parameter $\Sigma$, for the kernel-based test function class. 
\myeoe
\end{exa}

\begin{exa}[Two-parameter Gamma kernel] \label{exa:gamma}
 Consider the gamma distribution with density $f(x|\alpha,\beta) = \frac{\beta^\alpha}{\Gamma(\alpha)}x^{\alpha-1}e^{-\beta x}\1ve_{(0,\infty)}(x)$ with $\theta=(\alpha,\beta)\in \Theta=\{(\alpha,\beta)|\alpha>0,\beta>0\}$, and $\Gamma(x)$ denotes the Gamma function. 
Then 
\begin{equation}
\label{eqn:pde-gamma}
\frac{\partial}{\partial \beta}f(x|\alpha,\beta) = \frac{\alpha}{\beta}f(x|\alpha,\beta) - \frac{\alpha}{\beta}f(x|\alpha+1,\beta),
\end{equation}
This means that $f(x \mid \theta)$ violates first-order identifiability. Moreover, the interaction is dependent on parameter values across distinct
components. In other words, even the first-order strong identifiability condition is violated. One cannot appeal to Lemma \ref{lem:inverseboundfixoneargumentexactfitted} to obtain a pointwise inverse bound such as \eqref{eqn:localinverseboundfixargumentexactfitted}. 

In fact, such a bound is not achievable.
Fix $k_0\geq 2$ and consider $\mathcal{G}\subset \Ec_{k_0}(\Theta^{\circ})$, the subset of probability measures that have exactly $k_0$ atoms in $\Theta^\circ$, as
\begin{align*}
\mathcal{G}:= \{G\in \Ec_{k_0}(\Theta)| &G = \sum_{i=1}^{k_0}p_i\delta_{\theta_i} \text{ and there exist } i\neq j \\
&\text{ such that } \theta_j-\theta_i=(1,0) \}.
\end{align*}
Consider any $G_0 \in \mathcal{G}$, it can be shown that \citep{ho2016convergence,wei2022convergence}:
$$
\liminf_{\substack{G\overset{W_1}{\to} G_0\\ G\in \Ec_{k_0}(\Theta)}} \frac{h(P_G,P_{G_0})}{W_1(G,G_0)}=\liminf_{\substack{G\overset{W_1}{\to} G_0\\ G\in \Ec_{k_0}(\Theta)}} \frac{V(P_G,P_{G_0})}{W_1(G,G_0)}=0.
$$
This implies that even if $V(p_G,p_{G_0})$ vanishes at a fast rate, $W_1(G,G_0)$ may not vanish as fast.  
%
%
The "problematic" set $\mathcal{G}$ is referred to as \emph{pathological} set of parameter values for which the
first-order identifiability conditions fail to hold. This also corresponds to the subset values of parameter $G$ for which the Fisher information matrix is singular \cite{ho2019singularity}.
\myeoe
\end{exa}

\begin{comment}
\begin{exa}[Gamma kernel] \label{exa:gamma}
Consider Gamma kernel:
\[
p(x \mid \theta) := p(x \mid a,b)
= \frac{b^{a}}{\Gamma(a)} x^{a-1} e^{-bx},
\qquad a>0,\; b>0.
\]
Then
\[
\frac{\partial}{\partial b} p(x \mid a,b)
= \frac{a}{b}\, p(x \mid a,b)
- \frac{a}{b}\, p(x \mid a+1,b).
\]
This means that $p(x \mid \theta)$ violates first-order identifiability. Moreover, the interaction is dependent on parameter values across distinct
components (note the presence of $a$ and $a+1$ in the location parameter). See \cite{ho2016convergence,wei2022convergence} for more details. 
\myeoe
\end{exa}
\end{comment}

\begin{exa}[Skew-normal kernel] Consider
\(
f(x \mid \theta) = f(x \mid \mu, v, m),
\)
where $\mu$ is a location parameter, $v$ is a scale parameter, and $m$ is a skewness parameter, with
\[
f(x \mid \mu, v, m)
:= \frac{2}{\sqrt{v}}\,
\psi\!\left(\frac{x-\mu}{\sqrt{v}}\right)
\Psi\!\left(\frac{m(x-\mu)}{\sqrt{v}}\right).
\]
where $\psi$ and $\Psi$ in this example are probability density function and cumulative density function of standard normal distribution. The partial derivatives of the skew-normal kernel $f$ are highly dependent in multiple ways, as noticed in \cite{ho2019singularity}:
\begin{equation}
\label{eqn:pde-skewnormal}
\begin{cases}
\displaystyle
\frac{\partial^2 f}{\partial \mu^2}
- 2\,\frac{\partial f}{\partial v}
+ \frac{m^2 + m}{v}\,\frac{\partial f}{\partial m}
= 0, \\[1.2em]
\displaystyle
2m\,\frac{\partial f}{\partial m}
+ (m^2 + 1)\,\frac{\partial^2 f}{\partial m^2}
+ 2vm\,\frac{\partial^2 f}{\partial v\,\partial m}
= 0 .
\end{cases}
\end{equation}
This means that $f(x\mid\theta)$ violates second-order identifiability. The interaction among the parameters $\mu$, $v$, and $m$ occurs in different ways, and also depends on specific parameter values (such as when $m=0$). 
\myeoe
\end{exa}

Following \cite{ho2016convergence}, whenever either first or second-order identifiability conditions are violated, the mixture model is said to be in the setting of \emph{weak identifiability}. Due to the lack of strong identifiability conditions such as those formulated in subsection \ref{sec:strongpointwise}, we do not expect fast rates of estimation for the latent mixing measure such as the ones obtained by Theorem \ref{thm:convergenceratepointwise} (and its ensuing remarks).

\vspace{.1in}
\noindent \underline{Location-scale Gaussian mixtures} 
%
One of the first theoretical results for weakly identifiable mixture models were obtained for learning location-scale finite Gaussian mixtures, under the setting that only an upper bound on the number of mixture components $k$ is given. 
As before, the multivariate location-scale Gaussian kernel is denoted by $\{f(x \mid \theta, \Sigma)\}$ where $\theta$ ranges in a compact subset  $\Theta \subset \mathbb{R}^d$ and $\Sigma$ in a compact subset $\Omega \subset \mathbb{S}_d^{++}$ of the symmetric positive definite $d\times d$ matrices.
Given an $n$-i.i.d.\ sample $X_1,\ldots,X_n$ generated according to a Gaussian
mixture density
\[
p_{G^*}(x) = \int f(x \mid \theta, \Sigma)\, G^*(d\theta, d\Sigma),
\]
where the true mixing measure $G^*=G_0 = \sum_{i=1}^{k_0} p_i^0 \, \delta_{(\theta_i^0,\Sigma_i^0)}$ has
$k_0 \ge 1$ distinct support points, where $k_0 < k$.
Note that both $G^*$ and $k_0$ are unknown. 

Now, we shall overfit the data with a mixture of $k$ Gaussian distributions using the $n$-sample,
where $k \ge k_0 + 1$. Denote by
\(
\Gc_k := \Gc_k(\Theta \times \Omega)
\)
the set of probability measures on $\Theta \times \Omega$ with at most $k$
support points, and
\(
\mathcal{E}_{k_0} := \mathcal{E}_{k_0}(\Theta \times \Omega)
\)
the set of probability measures on $\Theta \times \Omega$ with exactly $k_0$
support points. In addition, given $c_0 \in [0,1)$, define a subset of
$\mathcal{E}_\ell$ by
\[
\mathcal{E}_{\ell,c_0}
:=
\left\{
  G = \sum_{i=1}^{\ell} p_i \delta_{(\theta_i,\Sigma_i)} \in \mathcal{E}_\ell
  :\ p_i \ge c_0 \ \forall\, 1 \le i \le \ell
\right\}.
\]
Define $\Gc_{k,c_0} = \cup_{\ell \in [k]} \mathcal{E}_{\ell,c_0}$. Of interest is the behavior of standard estimation methods such as the MLE or a Bayesian estimation procedure for $G$ restricted to the set of parameters $\Gc_{k,c_0}$.

Due to the lack of second-order identifiability of the Gaussian kernel in an overfitted setting, we do not expect a fast rate of convergence for the mixing measure, i.e. the rate $n^{-\frac{1}{4}}$ under $W_2$, according to Theorem \ref{thm:convergenceratepointwise}.
In fact, for the case of $k-k_0 = 1$, i.e., when the Gaussian mixture is overfitted by one extra component, the convergence rate $n^{-1/8}$ for the model parameters was established in a hypothesis test of heterogeneity \cite{chen2003tests}, or in fitting regression mixtures \cite{Shimotsu-2014}.

In general, it is shown by \cite{ho2016convergence} that the convergence rate of MLE for the mixing measure $G$ is of the order $n^{-1/(2\bar{r})}$ under the Wasserstein distance $W_{\bar{r}}$, modulo a log factor, where $\bar{r}$ is determined by the order of a set of polynomial equations.
Specifically, denote by $\bar{r} \ge 1$ the \emph{minimum} value of $r$ such that
the following system of polynomial equations:
\begin{equation*}
\sum_{j=1}^{k-k_0+1} \;
\sum_{n_1,n_2}
\frac{c_j^{\,2}\, a_j^{\,n_1} b_j^{\,n_2}}{n_1!\,n_2!}
= 0,
\qquad \text{for each } \alpha = 1,\ldots,r,
\end{equation*}
does \emph{not} have any nontrivial solution for the unknowns
$\{(a_j,b_j,c_j)\}_{j=1}^{k-k_0+1}$.
The ranges of $n_1,n_2$ in the second sum are all natural pairs satisfying
$n_1 + 2n_2 = \alpha$.
A solution is considered nontrivial if all of the $c_j$’s are nonzero, while at
least one of the $a_j$’s is nonzero.
In particular, if $k-k_0=1$ then $\bar{r}= 4$, but if $k-k_0=2$, $\bar{r}=6$. If $k-k_0 \geq 3$, then $\bar{r} \geq 7$. It is expected, although there is no proof, that $\bar{r} \uparrow$ as the amount of overfitting $k-k_0$ increases \citep{ho2016convergence}.

At the heart of this theory is the derivation of the pointwise overfitted inverse bound
 \begin{equation}
 \label{eqn:inversegauss}
    \liminf_{ \substack{G\overset{W_1}{\to} G_0\\ G\in \Gc_{k,c_0} }}
    \frac{V(P_G, P_{G_0})}{W_{\bar{r}}^{\,\bar r}(G, G_0)}
    > 0 .
\end{equation}
Moreover, the order $\bar{r}$ in the inverse bound cannot be improved, in the sense that if we replace $\bar{r}$ by any $r\in [1,\bar{r})$ the left hand side of Eq. \eqref{eqn:inversegauss} will be zero. Once this inverse bound is established, one may transfer the convergence behavior of the density estimate for $P_{G}$ into that of the mixing measure. This is made precise by the following (Theorem 2.1 of \cite{ho2016convergence}):

\begin{thm}
Let $\bar r$ be the singularity level as defined above, and $G_0\in \Ec_{k_0,c_0}$ for some $c_0>0$. Let $L, \underline{\lambda}, \bar\lambda$ be fixed positive numbers.
Given
\(
\Theta = [-L, L]^d
\) 
and let $\Omega$ be a subset of $\mathbb{S}_d^{++}$ whose eigenvalues are bounded
in the interval $[\underline{\lambda}, \bar{\lambda}]$.
Let $\hat G_n$ be the maximum likelihood estimate ranging in
$\mathcal{O}_{k,c_0}$.  Then
\begin{align*}
&\Pb_{G_0}
\Bigl(
  W_{\bar{r}}(\hat G_n, G_0)
  >
  C (\log n / n)^{1/(2\bar r)}
\Bigr) \\
\;\lesssim\; &
\exp(-c \log n).
\end{align*}
The constants $C$ and $c$ are positive and depend only on
$d, L, \underline{\lambda}, \bar\lambda, c_0$, and $G_0$.
\end{thm}

This result shows that pointwise convergence rates for standard estimators such as MLE may deteriorate rather quickly when working with an excessively overfitted mixture models. It is worth noting that the pointwise convergence rate of order root-$n$ may still be achieved, up to a logarithmic factor, by applying a suitable post-processing procedure \citep{do2024dendrogram}. 
\myredvsec{Although some minimax lower bounds were obtained by \cite{ho2016convergence} for the overfitted setting of location-scale Gaussian mixtures, the optimal minimax rate remains unknown.}

\subsection{Inverse bounds for singular mixture models}
\label{sec:singular}

The violation of first-order identifiability condition in the exact-fitted setting, and of the second-order identifiability condition in the overfitted setting is the indicator of the lack of non-singularity for the corresponding Fisher information matrix of the true model. When the Fisher information is non-singular, classical asymptotic theory tells us that a generic fast (root-$n$) rate of parameter estimation may be attained by a variety of generic estimation methods \citep{van2000asymptotic}. When the Fisher information is singular, fast rates of parameter estimation may no longer be expected for parametric model classes under many standard estimators. The overfitted mixture models with location-scale Gaussian kernels, and mixture models of two parameter Gamma kernels with pathological parameter values as described in the previous section are instances of \emph{singular} Fisher information  mixture models. These are examples where we know exactly where the singular points are; such information can be used to derive concrete rates of estimation for the model parameters. For instance, for location-scale Gaussian mixtures, singularities arise due to the overfitting, and the rates of parameter estimation depend on how much the level of overfitting $k-k_0$ is. The actual rate is dictated by $\bar{r}$ that given in the previous section. \myredvsec{The integer} $\bar{r}$ is referred to as a \emph{singularity level}, a general notion of singularity structure formalized in \cite{ho2019singularity}.

Singularity levels are just one of several structural concepts that allow us to describe, roughly speaking, \emph{how singular}  a singular model is. Such notions are devised to delineate the complexity of the parameter space (of mixing measures) based on which one may still be able to obtain fine details about the convergence behavior for parameter estimation.
In general, it can be found that different singular models carry different convergence behaviors; there are no longer generic rates of convergence. In fact, even different parameters within a singular model may possess different convergence behaviors, due to a kind of inhomogeneity. Moreover, different ranges of values for a single parameter may also receive different rates of convergence. This complex story was illustrated in great details for the skew-normal mixtures by \cite{ho2019singularity}, who also introduced a general method for deriving pointwise inverse bounds, where optimal transport based distances play a very useful role in expressing the aforementioned complexity. For the remainder of this section we shall provide a sketch of this theory.

\vspace{.1in}
\noindent \underline{Formalizing singularity structure}
Starting from the aforementioned partial differential structures expressed by identities such as \eqref{eqn:pde-gauss}, \eqref{eqn:pde-gamma}, or \eqref{eqn:pde-skewnormal}, we seek to represent the likelihood function $p_G(x)$ (as a function of $G$) in terms of linearly independent functions, which lead to a minimal form representation~\cite{ho2019singularity}. The minimal forms provide a basis for
studying the behavior of the likelihood as $G$ varies in a suitable neighborhood of mixing measures metrized by a suitable optimal transport distance. In particular, the singularity level of a mixing measure $G$ describes in a precise manner the variation of the mixture likelihood $p_G(x)$ with respect to changes in mixing measure $G$. Now, Fisher information singularities simply correspond to points in the parameter space which identify a mixing measure whose singularity level is non-zero. Within the subset of singular points, the parameter space can be partitioned into disjoint subsets determined by different singularity levels, 1, 2, 3,... and so on.

Suppose that a point in the parameter space corresponds to a mixing measure $G_0$ carrying singularity level $r$, then one can obtain \emph{pointwise inverse bounds} relating the Wasserstein distance in mixing measure to the Hellinger distance between corresponding mixture densities (cf. Theorem 3.1 and 3.2 of~\cite{ho2019singularity}): For any $s\geq r+1$, and fixed $G_0$
\begin{equation}
W_{s}^s(G,G_0) \lesssim h(p_G, p_{G_0})
\label{inverse}
\end{equation}
holds for any $G$ such that the left hand side is sufficiently small. Moreover, the exponent $s$ in the bound cannot be improved. 
An immediate consequence of such an inverse bound  is the impact of the singularity level on rates of parameter estimation.
Given an i.i.d.
$n$-sample from a (true) mixture density $p_{G_0}$, where $G_0$ admits a singularity
level $r$. This will entail under some mild conditions on $f$ that a
standard estimation method such as maximum likelihood estimation and Bayesian estimation
with a non-informative prior yields root-$n$ rate of convergence for the mixture density $p_G$, which
induces the rate of convergence $n^{-1/2(r+1)}$ for the mixing measure $G$, which is also a minimax lower bound
(up to a logarithmic factor), under the suitable Wasserstein metric $W_{r+1}$.
Thus, singularity level 0 results in root-$n$ convergence rate for mixing measure estimation.
Fisher singular points corresponding to singularity level 1, 2, or 3 exhibit
convergence rates $n^{-1/4}, n^{-1/6}, n^{-1/8}$ or so on.

\vspace{.1in}
\noindent \underline{Handling mixed parameter types}
The singularity level is defined for the entire collection of parameters encapsulated by the mixing measure $G$. As illustrated for location-scale Gaussian mixtures, where we used the notation $\bar{r}$, they relate to a notion of dimensionality for some \emph{semi-algebraic} sets, i.e., sets that can be defined as solutions of a system of polynomial equations combined with some inequalities. However, to anticipate the inhomogeneity of parameters of different types (e.g., location, scale, skewness) that arise from the inhomogeneity of such polynominal systems, a refined notion of singularity structure is useful: a vector-valued \emph{singularity index}, which extends the notion of natural-valued singularity level described
earlier. The singularity index describes the variation of the mixture likelihood with
respect to changes of individual parameters of each type. Moreover, pointwise inverse bounds can be established in terms of a generalized optimal transport distance, where the cost of transportation is defined in terms of a semi-metric in which each type of parameters contributes differently according to its corresponding singularity index. 
As a result, a singularity index $\kappa$ corresponds to singularity level $(\|\kappa\|_\infty-1)$ for
the mixing measure, which may admit the rate of convergence $n^{-1/2\|\kappa\|_\infty}$, but the individual parameter of type $j$ may admit the rate of convergence $n^{-1/2\kappa_{j}}$ instead.
We refer the reader to \cite{ho2019singularity} for the full details.

\section{Inference of de Finetti's mixing measure}

From this point on, we depart the basic mixture modeling land to venture into a much larger realm of latent structured models. The range of such probabilistic modeling creatures one may encounter is vast, as they have been invented and evolving over time to accommodate increasingly complex data domains and evermore ambitious inferential tasks by statistical modelers working in a variety of fields. In this article, for a number of reasons we will primarily be concerned with hierarchical models. First, they cover a large number of statistical models that have been successfully deployed in real-world applications, as mentioned in the Introduction. Second, hierarchical models represent the principal modeling toolbox for statistical modelers, especially (but not exclusively so) the practitioners of Bayesian analysis. The third reason is about a conceptual and theoretical appeal: a hierarchical model may be viewed as an abstract version of mixture modeling, as it is a mixture of mixture distributions. This allows one to inherit many valuable insights and theory developed for the basic mixture building blocks to address questions that arise from hierarchical model-based inference. The interesting conceptual leap from the basic mixture modeling framework is that hierarchical model based inference is often motivated from the need or the opportunity of working with not one, but multiple data population samples. 

\subsection{de Finetti's theorem and implications}
\label{sec:deFinetti}
It may be difficult to find a single organizing principle for hierarchical modeling, but the representation theorem of Bruno de Finetti's on 
exchangeable random variables is close to being one. This theorem states roughly that, if $X_1, X_2,\ldots$ is an infinite exchangeable sequence of random variables defined in a Borel measure space $(\Xf, \Xc)$, then there exists a random variable $\theta$ in some space $\Theta$, where $\theta$ is distributed according to a probability distribution $G$, such that $X_1,X_2,\ldots$ are conditionally i.i.d. given $\theta$ \citep{aldous1985exchangeability,kallenberg2006probabilistic}.
In particular, let $f( x|\theta)$ denote the conditional density of $X_i$ given $\theta$, we may express the joint density of an exchangeable $N$-sample $X_{[N]}:=(X_1,\ldots,X_N)$, for any $N\geq 1$, as a mixture of product distributions taking the following form: 
\begin{equation}
\label{eqn:definetti}
f_N(x_1, \ldots,  x_N) = \int\prod_{n=1}^{N} f(x_n|\theta) G(\dd \theta).
\end{equation}
In this representation, the probability measure $G$ is known as de Finetti's mixing measure. An important ramification of this representation theorem is that it provides a foundation for conditional i.i.d. sampling models, via the product density $\prod_{n=1}^{N}f( x_n|\theta)$, and enables the Bayesian interpretation of the random parameter $\theta$, where the mixing measure $G$ plays the role of the a prior probability distribution, which expresses a Bayesian uncertainty about $\theta$, which parameterizes the unknown i.i.d. density $f(x|\theta)$.

The second ramification of de Finetti's theorem, and one which we will expand on for the remainder of this article, stems from the interest in learning about de Finetti's mixing measure $G$ from observed data. In particular, one may treat $G$ as a probabilistic model for the assumed heterogeneity about data populations of exchangeable sequences. In order to obtain an estimate of the mixing measure $G$, one needs not one but multiple copies of the exchangeable sequences $X_{[N]}$. A simplest starting point is to assume that we are given $m$ \emph{independent} copies of exchangeable sequences denoted by $\{X_{[N]}^{i}\}_{i=1}^{m}$, where $i=1,\ldots,m$ provides the index for the $i$-th copy of the exchangeable sequence. We have arrived at a basic instance of a hierarchical model, in which the latent probability measure $G$ represents an object of inferential interest, and the connection of $G$ to the observed data is (equivalently to \eqref{eqn:definetti}) expressed via the customary two-stage hierarchical specification as follows
\begin{align}
\label{eqn:mixtureproduct2}
\theta_1,\ldots, \theta_m | G & \stackrel{iid}{\sim}  G, \\
X_1^{i},\ldots, X_{N}^{i} | \theta_i & \stackrel{iid}{\sim}  f(\cdot |\theta_i), \; i=1,\ldots, m. \nonumber
\end{align}
An important departure from the basic setting of a classical statistical theory is that we do not have just a single i.i.d. sample of a data population, but in fact there are $m$ such samples available. To be precise, we are given $m$  independent copies of $N$-sequences, each sequence is made of $N$ exchangeable observations. In total, there is a $m\times N$ data set.
Clearly, if $N=1$, i.e., each sampled sequence has only one observation, then we are reduced to a $m$-i.i.d. sample of observations, which are distributed according to the (marginal) density $ \int f(x_1|\theta) G(\dd \theta)$, the familiar mixture model's representation. Thus, one important aspect of a statistical theory for the mixture of product distributions is the identifiability of the mixing measure $G$, as well as the \emph{distinct} roles of data sample dimensions $m$ and $N$ on the efficiency of statistical estimation. When $G$ is a discrete probability measure with a known number of atoms, then a rather complete statistical theory can be had and will be described in subsection \ref{sec:mixproduct}. However, much remains open beyond this simple setting.

It is often the case that advances in novel hierarchical modeling happen at a much faster pace than the progress one can make regarding the theoretical behavior of inferential methods arising from such modeling tools. Continuing with our interest in learning from an $m\times N$ data set, however, in many applications there is no reason to restrict ourselves to assuming that the $m$ sequences $\{X_{[N]}^{i}\}_{i=1}^{m}$ are independent copies. A natural relaxation in the modeling is to posit that these $m$ sequences are exchangeable only. The theorem of de Finetti's continues to apply, which entails that the $m$ sequences of observations are conditionally i.i.d. --- each of the $m$ sequences may be associated with a mixing measure $G_i$, for $i=1,\ldots, m$, because it is made of exchangeable observations, where the $G_i$ are conditionally i.i.d. according to some distribution $\Dscr$, given $\Dscr$, which is itself random. We write $\Dscr \sim \Pi$. This hierarchical specification is an extension of the two-stage model given above:
\begin{align}
\label{eqn:hierarchical}
\Dscr & \sim \Pi, \nonumber \\
G_1,\ldots, G_m | \Dscr & \stackrel{iid}{\sim} \Dscr, \\
X_1^{i},\ldots, X_{N}^{i} | G_i & \stackrel{iid}{\sim}  \int f(\cdot |\theta) G_i(\dd \theta), \; i=1,\ldots, m.
\nonumber
\end{align}
Hierarchical models for an exchangeable collection of exchangeable sequences of observed data, such as the one just described, were eloquently motivated via de Finetti's theorem in \cite{blei2003latent}. A parametric instance of \eqref{eqn:hierarchical} is known as Latent Dirichlet allocation model, which was initially motivated from the analysis of corpora of text documents, and subsequently adopted in a variety of data domains. The LDA is colloquially known as a "bag-of-word" model, which highlights the modeling assumption that each sampled document is viewed as a "bag" of sampled words, which are exchangeable; a more complete description should evoke the image of a loosely made container to carry a large exchangeable collection of such bags. In population genetics, an equivalent probabilistic model was independently introduced and called "admixture model" \citep{pritchard2000inference}. Like the LDA, it is immensely popular. In the past two decades the LDA/ finite admixture model have inspired numerous extensions and modeling developments, especially in the field of Bayesian nonparametrics, with contributions from many authors, see \citep{Teh-etal-06,Rodriguez-etal-08,Nguyen-10,camerlenghi2019distribution,catalano2024unified} and additional references therein. The asymptotic theory on these more complex hierarchical models are less developed. In Section \ref{sec:hierarchical} we shall provide a brief overview and highlight how optimal transport continues to play important roles on such a theory.

\subsection{Mixture of product distributions}
\label{sec:mixproduct}

Now, we shall zoom in the mixture of product distributions setting introduced in \eqref{eqn:definetti} and \eqref{eqn:mixtureproduct2}, which may be viewed as a simplest instance of a hierarchical model. Recall that
the data is given as a collection of $m$ sequences denoted by $X^i_{[\m]}:= (X_1^i,\ldots,X_{\m}^i)$ for $i=1,\ldots,\n$. 
The sequences are composed of elements in a measurable space $(\Xfrak,\Acalcal)$. Examples include $\Xfrak = \R^d$, $\Xfrak$ is a discrete space, and $\Xfrak$ is a space of measures. 
In any case, in this section the latent de Finetti's mixing measure will be assumed to be a discrete probability measure
with a finite number of supporting atoms on $\Theta$, which is assumed to be a subset of $\Rb^q$. 
\myredvsec{Recall that} $\Ec_k(\Theta)$ stand for the space of all probability measures on $\Theta$ that have exactly $k$ atoms.
 

Write $P\otimes Q$ to be the product measure of $P$ and $Q$ and $\otimes^N P$ for the $N$-fold product of $P$. For each $N\geq 1$ the $N$-product probability kernel $P_{\theta,N}(\dd x_1,\ldots \dd x_n) :=\bigotimes_{n=1}^{\m}P_\theta(dx_n)$ represents the product measure on \\  $(\Xfrak^{\m},\Xc^{\m})$, where $\Xc^{\m}$ is the product sigma-algebra.  Without loss of generality, assume that the map 
$\theta\mapsto \P_\theta $ 
is injective; the probability kernel uniquely determines its parameter $\theta$. 
Suppose that the mixing measure $G \in \Ec_{k}(\Theta)$. Write $G=\sum_{i=1}^{k} p_i \delta_{\theta_i}$. 
Thus we may rewrite Eq.~\eqref{eqn:definetti} as follows,
\begin{equation}
\label{eqn:mixprod}
\P_{G,\m} = \sum_{i=1}^{k} p_i P_{\theta_i,\m},
\end{equation}
where the subscripts $G$ and $\m$ highlight the mixing measure, and the $N$-product distributional components. 
Now, given $m$ independent copies of exchangeable sequences $X^i_{[N]} = (X_1^i,\ldots,X_{\m}^i)$, for $i=1,\ldots,\n$, each of which is distributed according to $P_{G,N}$ given above, where the conditional probability kernel $f$ is given and known. The questions of interest is the identifiability of the mixing measure $G$, and how efficiently can one estimate $G$ from the $m\times N$ data set.  

As noted earlier, when $\m=1$, \eqref{eqn:mixprod} reduces to a standard finite mixture distribution $P_{G,1} = \sum_{j=1}^{k}p_j P_{\theta_j}$. For general $N$, it is expected that the larger $N$ is, the more information each data sequence contains, hence the more easily one may be able to identify and estimate the mixing measure $G$. 
We shall take up these issues in turn. 

\vspace{.1in}
\noindent
\underline{Identifiability} 
Before considering the question of convergence rate for estimating $G$, we need to make sure that $G$ is identifiable.
This question has occupied the interest of a number of authors ~\cite{teicher1967identifiability,elmore2005application,hall2005nonparametric}, with decisive results obtained by~\cite{allman2009identifiability} on finite mixture models for conditionally independent (but not necessarily identically distributed) observations. Note that our setting here is restricted to finite mixtures of conditionally i.i.d. observations, where identifiability results were obtained by \cite{vandermeulen2019operator}. In particular,
the model $\{\P_{G,\m}\}_G$ is said to be identifiable on $\Gc_k(\Theta)$ if the map $G\mapsto P_{G,\m}$ is injective on $\Gc_k(\Theta)$. As already discussed, one should expect the larger the $N$ is, the easier the identifiability holds.
The next theorem make this intuition precise.  

\begin{thm}
\label{thm:idenmixpro}
The model $\{\P_{G,\m}\}_G$ is identifiable on $\Gc_k(\Theta)$ as soon as $N \ge 2k-1$.
\end{thm}

Results in more specific settings have been reported by various authors; the general theorem given this way is due to \cite[Theorem 4.1]{vandermeulen2019operator}. In fact, a stronger conclusion than identifiability can be drawn if taking more products. Let $\Gc_\infty(\Theta):=\cup_{\ell=1}^\infty \Ec_\ell(\Theta)$. 

\begin{thm}
\label{thm:idenmixpro2}
Fix $G\in \Gc_k(\Theta)$. When $N \ge 2k$, there does not exist a distinct $G' \in \Gc_\infty(\Theta) $ such that $\P_{G,\m}=\P_{G',\m}$. 
\end{thm}

The above result is reported in \cite[Theorem 4.3]{vandermeulen2019operator}, where the term ``determined'' is used. Moreover, if one additionally knows that the component probability measure are linear independent, then the products needed to guarantee identifiability may be significantly decreased. 

\begin{thm}\label{thm:idenmixpro3}
Fix $G=\sum_{i=1}^{k} p_i \delta_{\theta_i} \in \Ec_{k}(\Theta)$. Suppose that $P_{\theta_1},\ldots,P_{\theta_k}$ are linear independent. Then when $N \ge 3$, there does not exists a distinct $G'\in \Gc_k(\Theta)$ such that $\P_{G,\m}=\P_{G',\m}$. Moreover, when $N \ge 4$, there does not exists a distinct $G'\in \Gc_\infty(\Theta)$ such that $\P_{G,\m}=\P_{G',\m}$.
\end{thm}

The above result is a combination of  \cite[Theorems 4.5 and 4.6]{vandermeulen2019operator}. The above identifiability results are also known to correspond to the notion of "strict identifiability", as opposed to "generic identifiability", i.e., identifiability of the parameter $G$ up to a measure-zero set. For the latter notion of identifiability and additional
results for general mixture of product distributions, which may be obtained by applying Kruskal’s theorem, see \cite{allman2009identifiability}.

\vspace{.1in}
\noindent
\underline{Inverse bounds} 
In parallel to the approach and results described in the previous sections, it is possible to establish a collection of inverse bounds for mixtures of product distributions, which will prove to be a useful tool for obtaining a precise characterization the convergence behavior of estimates for the mixing measure $G$. 
Here, we consider the case in which $G \in \Ec_{k}(\Theta)$, where $k$ is known.

Before proceeding further, let us pause and consider what kind of convergence behavior that we anticipate for the mixing measure $G$, given the $m\times N$ data. $G$ is composed of $k$ atoms $\theta_i$, each of which is associated with the mixing probabilities $p_i$, for $i=1,\ldots, k$. Because each sampled sequence will be an i.i.d. sample from one of the product distributions $P_{\theta_i,N}$. If $N$ increases we should expect a reduced estimation error regarding the product distribution components, which should help with the estimation of one of the atoms $\theta_i$, for $i=1,\ldots, k$. In a parametric setting, we anticipate a rate that decreases at the order of $N^{-1/2}$. However, we do not expect that increasing $N$ while keeping $m$ fixed will help to drive down the estimation error for the mixing probabilities $p_i$ at all.
That would be the case, only if $m$ also increases, so that each of the $k$ components will receive their shares of representative sampled sequences.
In other words, we anticipate the asymmetric roles that sample size $m$ and $N$ play in the estimation of different parameters of the model, with the atoms benefit more from the sequence length $N$. 

To emphasize the role of $N$ in the denominator of the inverse bound, a more useful refinement than the standard Wasserstein distance $W_r$ is the following. For any $G=\sum_{i=1}^kp_i\delta_{\theta_i} \in \Ec_k(\Theta)$ and $G'=\sum_{i=1}^kp'_i\delta_{\theta'_i} \in \Ec_k(\Theta)$, define
 $$
 \myD_{\m}(G,G') =  \min_{\tau\in S_{k}} \sum_{i=1}^{k}(\sqrt{\m}\|\theta_{\tau(i)}-\theta'_i\|_2+|p_{\tau(i)}-p'_i|) 
 $$ 
 where $S_{k}$ denote all the permutations on the set $[k]$.

 It is simple to verify that $\myD_{\m}(\cdot,\cdot)$ 
is a valid metric on $\Ecal_{k}(\Theta)$ for each $\m$ and is equivalent to a suitable optimal transport distance metric. Indeed, $G = \sum_{i=1}^{k}p_i \delta_{\theta_i} \in \Ec_k(\Theta)$, due to the permutations invariance of its atoms, can be identified as a set $\{ (\theta_i,p_i): 1\leq i \leq k\}$, which can further be identified as $\tilde{G} = \sum_{i=1}^{k}\frac{1}{k} \delta_{(\theta_i,p_i)}\in \Ec_{k}(  \Theta \times \R)$. Formally, we define a map $\Ec_k(\Theta)\to \Ec_k( \Theta\times \R) $ by
\begin{equation}
G = \sum_{i=1}^{k}p_i \delta_{\theta_i} \mapsto \tilde{G} = \sum_{i=1}^{k}\frac{1}{k} \delta_{(\theta_i,p_i)}\in \Ec_{k}( \Theta \times \R). \label{eqn:a}
\end{equation}
Now, endow $ \Theta \times \R$ with a metric $M_{\m}$ defined by $M_{\m}((\theta,p), (\theta',p')) = \sqrt{\m}\|\theta-\theta'\|_2+|p-p'|$ and note the following fact.

\begin{lem}\label{lem:rotationinvariance}
	For any $ G=\sum_{i=1}^k\frac{1}{k}\delta_{\bar{\theta}_i},G'=\sum_{i=1}^k\frac{1}{k}\delta_{\bar{\theta}'_i} \in \Ec_k(\bar{\Theta})$ and distance $d_{\bar{\Theta}}$ on $\bar{\Theta}$, 
	$$W_p^p(G,G'; d_{\bar{\Theta}}) =  \min_{\tau\in S_{k}} \frac{1}{k}\sum_{i=1}^k d_{\bar{\Theta}}^p(\theta_i, \theta'_{\tau(i)}).$$
\end{lem}

The above characterization of Wasserstein distance for equally-weighted discrete distributions is available in \cite[Proposition 2]{nguyen2011wasserstein} and \cite[Lemma 3.1]{wei2022convergence}.  By applying Lemma \ref{lem:rotationinvariance} with $\bar{\Theta}$, $d_{\bar{\Theta}}$ replaced respectively by $\Theta\times \R$ and  $M_\m$, then for any $ G,G' \in \Ec_k(\Theta)$, $W_1(\tilde{G},\tilde{G'}; M_{\m})= \frac{1}{k} D_{\m}(G,G') $, which validates that $D_{\m} $ is indeed a metric on $\Ec_k(\Theta)$, and moreover it does not depend on the specific representations of $G$ and $G'$. It is  straightforward that $W_1(G,G_0)\asymp D_1(G,G_0)$ in a small neighborhood of $G_0$.

Due to the nature of the product structure, it is of interest to investigate the smallest number of products such that identifiability or inverse bounds hold.    
Therefore define, for any $G_0 \in \Ec_{k_0}(\Theta^\circ)$, where $\Theta^\circ$ denotes the interior of $\Theta$, $ H \in \cup_{k=1}^\infty \Ec_k(\Theta)$ and $\Hc\subset \cup_{k=1}^\infty \Ec_k(\Theta)$, 

     \begin{align*}
      \label{eqn:defn0n1n2}
      & n_0 :=  n_0(H, \Hc ) \\
      & :=  \min \biggr \{n \geq 1 \biggr | \forall G \in \Hc \setminus \{H\}, P_{G,n} \neq P_{H,n} \biggr \} ,
    \\
    & n_1 := n_1(G_0)  := n_1(G_0,\Ec_{k_0}(\Theta)) \\
    & :=  \min \biggr \{n \geq 1 \biggr | \liminf_{\substack{G\overset{W_1}{\to} G_0 \nonumber \\ G\in \Ecal_{k_0}(\Theta) }} \frac{V(P_{G,n },P_{G_0,n })}{\myD_1(G,G_0)}>0 \biggr \}. 
    \nonumber
     \end{align*}
$n_0$ is called minimal zero-order identifiable length (with respect to $H$ and $\Hc$); 
$n_1$ is called minimal first-order identifiable length (with respect to $G_0$ and $\Ec_{k_0}(\Theta)$). 
Since $W_1(G,G_0)\asymp D_1(G,G_0)$ in a small neighborhood of $G_0$, 
the two metrics ar interchangeable in the denominator of the above definition for $n_1$. 
When it is clear from the context, we may use $n_1$ or $n_1(G_0)$ for $n_1(G_0,\Ec_{k_0}(\Theta))$ for brevity. Similar rules apply to $n_0$. 
The next example illustrates these concepts and the usefulness of overcoming weak identifiability in mixture models by taking products. 

\begin{exa}[Continuation on two-parameter gamma kernel] 
\label{exa:gamma2} 
Consider the gamma distribution $f(x|\alpha,\beta)$ discussed in Example \ref{exa:gamma}. Let $k_0\geq 2$ and recall that $\Gc$ denotes the pathological subset of the gamma mixture's parameter space, such that for any $G_0 \in \mathcal{G}$
$$
\liminf_{\substack{G\overset{W_1}{\to} G_0\\ G\in \Ec_{k_0}(\Theta)}} \frac{h(P_G,P_{G_0})}{D_1(G,G_0)}=\liminf_{\substack{G\overset{W_1}{\to} G_0\\ G\in \Ec_{k_0}(\Theta)}} \frac{V(P_G,P_{G_0})}{D_1(G,G_0)}=0.
$$
This means $n_1(G_0)\geq 2$ for $G_0\in \mathcal{G}$.   
In fact it is shown by \cite[Example 5.11]{wei2022convergence} that $n_1(G_0)= 2$ for $G_0\in \mathcal{G}$. That is, by taking reparative product of $2$, the weak identifiability for mixture of Gamma distribution is overcome. It also follows from Lemma \ref{lem:inverseboundfixoneargumentexactfitted} that for any $G_0\in \Ec_{k_0}(\Theta)\backslash \mathcal{G}$, $n_1(G_0)=1$. 
One can show that $\{f(x|\theta_i)\}_{i=1}^k$ are linear independent for any distinct $\theta_1,\ldots,\theta_k\in \Theta$ for any $k$. The linear independence immediately implies that $p_G$ is identifiable on $\Gc_\infty(\Theta)$, i.e. for any $G\in \Ec_k(\Theta)$ and any $G'\in \Ec_{k'}(\Theta)$, $P_G\not = P_{G'}$. Thus, $n_0(G,\Gc_\infty(\Theta))=1$ 
for any $G\in \Gc_\infty(\Theta)$. \myeoe 
\begin{comment}
\YW{Add mixture identifiability?} \LN{sure, since you are at it}
\end{comment} 
\end{exa}

\begin{comment}
The following proposition provides the link between classical identifiability and strong notions of local identifiability provided either~\eqref{eqn:genthmcon} or~\eqref{eqn:curvatureprodbound} holds. In a nutshell, as $\m$ gets large, the two types of identifiability can be connected by the force of the central limit theorem applied to product distributions, which is one of the key ingredients in the proof of the inverse bounds.
%
%
Define two related and useful quantities: for any $G_0\in \Ec_{k_0}(\Theta^\circ)$
\begin{align}
\label{eqn:Nbar}
\underbar{\m}_1 := \underbar{\m}_1(G_0) := \min \biggr \{n \geq 1 \biggr | \inf_{N\geq n} \liminf_{\substack{G\overset{W_1}{\to} G_0\\ G\in \Ecal_{k_0}(\Theta) }} \frac{V(P_{G,\m },P_{G_0,\m })}{\myD_{\m}(G,G_0)}>0 \biggr \} \\
\label{eqn:N2bar}
\underbar{\m}_2  := \underbar{\m}_2(G_0) :=   \min \biggr \{n \geq 1 \biggr | \inf_{\m\geq n} \lim_{r\to 0}\ \inf_{\substack{G, H\in B_{W_1}(G_0,r)   \\ G\not = H }} \frac{V(\P_{G,\m},\P_{H,\m})}{\myD_{\m}(G,H)}>0 \biggr \}.
\end{align}
Note that~\eqref{eqn:genthmcon} means $\underbar{\m}_1(G_0) < \infty$, while~\eqref{eqn:curvatureprodbound} means $\underbar{\m}_2(G_0) < \infty$. 
\end{comment}

Finally let us present the desired inverse bound for mixtures of product distributions, which emphasizes the dependence on $N$: 
for a fixed $G_0 \in \Ecal_{k_0}(\Theta^\circ)$ there holds
\begin{equation}
\liminf_{\m\to \infty}\liminf_{\substack{G\overset{W_1}{\to} G_0\\ G\in \Ecal_{k_0}(\Theta) }} \frac{V(\P_{G,\m},\P_{G_0,\m})}{\myD_{\m}(G,G_0)} > 0. \label{eqn:genthmcon}
\end{equation}
%
\begin{comment}
By contrast, an easy upper bound on the left side of~\eqref{eqn:genthmcon} holds generally (cf. Lemma~\ref{lem:VD1liminfuppbou}):
\begin{equation}
\sup_{\m \geq 1}\liminf_{\substack{G\overset{W_1}{\to} G_0\\ G\in \Ecal_{k_0}(\Theta) }} \frac{V(\P_{G,\m},\P_{G_0,\m})}{\myD_{\m}(G,G_0)} \leq 1/2. \label{eqn:liminfliminfVDNup}
\end{equation}
\end{comment}
In fact, a strong inverse bound can also be established:
\begin{equation}
\liminf_{\m\to\infty} \ \liminf_{\substack{G\neq H\in \Ec_{k_0}(\Theta)\\ G, H \overset{W_1}{\to} G_0 }} \frac{V(\P_{G,\m},\P_{H,\m})}{\myD_{\m}(G,H)}>0. \label{eqn:curvatureprodbound}
\end{equation}
It's easy to see that \eqref{eqn:curvatureprodbound} implies \eqref{eqn:genthmcon}. 
Next we present some probability kernels such that the above inverse bounds hold. 

\vspace{.1in}
\noindent
\underline{Probability kernels in exponential family}
%
Exponential families of distributions represent one of the most common building blocks in a statistical modeler's toolbox.  We now present results establishing inverse bounds for the mixture of products of exponential family distributions. 

\begin{comment}
, where we switch the notation from $\theta$ to $\eta$. \LN{This change of notation is a bit distracting and not necessary in my opinion...}{\color{purple} I think it will come handy given the fact that exponential usually has the  two different set of parameters. $\eta$ is convenient for proving theorems while the usual parameter, which will be denoted by $\theta$, is for common use.}
\end{comment}

\begin{thm}
\label{thm:expfam} 
Suppose that the probability kernel has a density function $f(x|\theta)$ in a full rank exponential family, given in its canonical form $f(x|\thetaeta)=\exp(\langle \thetaeta,T(x)\rangle -A(\thetaeta))h(x)$ with $\thetaeta\in \Theta$, the natural parameter space.
For any $G_0\in \Ec_{k_0}(\Theta^\circ)$, \eqref{eqn:genthmcon} and \eqref{eqn:curvatureprodbound} hold.
\end{thm}

In the last theorem the exponential family is in its canonical form. The following theorem extends to exponential families in the general form under mild conditions. 

\begin{thm}
\label{cor:expfamtheta}
Consider the probability kernel that has a density function $f(x|\theta)$ in a full rank exponential family,  
$f(x|\theta)=\exp\left(\langle \eta(\theta),T(x)\rangle -B(\theta) \right)h(x),$
where the map $\eta:\Theta \to \eta(\Theta)\subset \R^q$ is a homeomorphism. Suppose that $\eta$ is continuously differentiable on $\Theta^\circ$ and its the Jacobian is of full rank on $\Theta^\circ$. Then, for any $G_0\in \Ec_{k_0}(\Theta^\circ)$, ~\eqref{eqn:genthmcon} and~\eqref{eqn:curvatureprodbound} hold.
\end{thm}

The inverse bounds \eqref{eqn:genthmcon} and \eqref{eqn:curvatureprodbound} may be established for fairly general kernels that go beyond exponential family, including the case where the data space $\Xf$ is abstract; see \cite[Sections~5.2 and 5.3]{wei2022convergence} for more details. We are now in a good position to present a convergence rate result for mixtures of product distribution based on Bayesian estimation procedure.

\vspace{.1in}
\noindent
\underline{Posterior contraction} 
%
	Suppose the data are $\n$ independent sequences of exchangeable observations, $X_{[\m_i]}^i = (X_{i1},X_{i2},\cdots,X_{i\m_i}) \in \Xfrak^{\m_i}$ for $i\in[\n]$. Each sequence $X_{[\m_i]}^i$ is assumed to be a sample drawn from a mixture of $\m_i-$product distributions $P_{G,\m_i}$ for some "true" de Finetti (mixing) measure $G^*=G_0 \in \Ec_{k_0}(\Theta)$. 
	The problem is to estimate $G_0$ given the $m$ independent exchangeable sequences.
	A Bayesian statistician endows $(\Ec_{k_0}(\Theta), \Bc(\Ec_{k_0}(\Theta)))$ with a prior distribution $\Pi$ and obtains the posterior distribution $\Pi(dG| X^1_{[\m_1]},\ldots, X^{\n}_{[\m_m]})$ by Bayes' rule, where $\Bc(\Ec_{k_0}(\Theta))$ is the Borel sigma algebra w.r.t. $D_1$ distance. We now study the asymptotic behavior of this posterior distribution as the amount of data $\n\times\m$ tend to infinity. 
	
	 Suppose throughout this section,  $\{P_{\theta}\}_{\theta\in \Theta}$ has density $\{f(x|\theta)\}_{\theta\in \Theta}$ w.r.t. a $\sigma$-finite dominating measure $\lambda$ on $\Xfrak$; 
	 then $P_{G,\m_i}$ for $G=\sum_{i=1}^{k_0}p_i\delta_{\theta_i}$ has density w.r.t. to $\lambda$:
	\begin{equation*}
	p_{G,\m_i}(\bar{x}) = \sum_{i=1}^{k_0} p_i\prod_{j=1}^{\m_i} f(x_j|\theta_i),  \text{for } \bar{x}=(x_1,x_2,\cdots,x_{\m_i}). 
   \end{equation*}
	Then the density of $X_{[\m_i]}^i$ conditioned on $G$ is $p_{G,\m_i}(\cdot)$.
	%
    \begin{comment}
	Since $\Theta$ as a subset of $\R^q$ is separable, $\Ec_{k_0}(\Theta)$ is separable. 
	Thus, the posterior distribution (a version of regular conditional distribution) is the random measure given by 
	\begin{equation}
	\Pi(B|X_{[\m_1]}^1,\ldots,X_{[\m_{\n}]}^{\n}) = \frac{\int_B\prod_{i=1}^\n p_{G,\m_i}(X_{[\m_i]}^i)d\Pi(G)}{\int_{\Ec_{k_0}(\Theta)}\prod_{i=1}^\n p_{G,\m_i}(X_{[\m_i]}^i)d\Pi(G)}, \label{eqn:bayesrule}
	\end{equation}
	for any Borel measurable subset of $B \subset \Ec_{k_0}(\Theta)$.	
	%
    \end{comment}
	It is customary to express the above model equivalently in the hierarchical Bayesian fashion: \begin{align*}
	&G \sim \Pi,  \quad \theta_1,\theta_2,\cdots,\theta_\n|G \overset{i.i.d.}{\sim} G\\
	&X_{i1},X_{i2},\cdots,X_{i\m_i} |\theta_i  \overset{i.i.d.}{\sim} f(x|\theta_i) \quad \text{for }i=1,\cdots,\n.
	\end{align*}
	As above, the $\n$ independent data sequences are denoted by $X_{[\m_i]}^i = (X_{i1},\cdots,X_{i\m_i}) \in \Xfrak^{\m_i }$ for $i\in[\n]$. 
The following assumptions are required for the main theorems of this section.
	
	\begin{enumerate}[label=(B\arabic*)]
	
	\item \label{item:prior} (Prior assumption)  There is a prior measure $\Pi_{\theta}$ on $\Theta_1\subset \Theta$ with its Borel sigma algebra possessing a density w.r.t. Lebesgue measure that is bounded away from zero and infinity, where $\Theta_1$ is a compact subset of $\Theta$. Define $k_0$-probability simplex $\Delta^{k_0}:=\{(p_1,\ldots,p_{k_0}) \in \R_+^{k_0} | 
	\sum_{i=1}^{k_0} = 1\}$, Suppose there is a prior measure $\Pi_p$ on the $k_0$-probability simplex possessing a density w.r.t. Lebesgue measure on $\R^{k_0-1}$ that is bounded away from zero and infinity. Then $\Pi_p\times \Pi_{\theta}^{k_0}$ is a measure on $$\{((p_1,\theta_1),\ldots,(p_{k_0},\theta_{k_0}))|p_i\geq 0, \theta_i\in \Theta_1, \sum_{i=1}^{k_0}p_i=1\},$$ which induces a probability measure on $\Ec_{k_0}(\Theta_1)$. 
	Here, the prior $\Pi$ is generated by via independent $\Pi_p$ and $\Pi_\theta$ and the support $\Theta_1$ of $\Pi_\theta$ is such that $G_0\in \Ec_{k_0}(\Theta_1)$.


	\item \label{item:kernel} (Kernel assumption) Suppose that for every $\theta_1,\theta_2\in \Theta_1$, $\theta_0\in \{\theta_i^0\}_{i=1}^{k_0}$ and some positive constants $\alpha_0,L_1, \beta_0, L_2$,
	 \begin{align}
	 K(f(x|\theta_0),f(x|\theta_2))\leq &L_1\|\theta_0-\theta_2\|_2^{\alpha_0}, \label{eqn:KLlipschitz} \\
	 h(f(x|\theta_1),f(x|\theta_2))\leq & L_2\|\theta_1-\theta_2\|_2^{\beta_0}. \label{eqn:hellingerlipschitz}
	 \end{align}
	\begin{comment}
	Suppose $K(f(x|\theta_1),f(x|\theta_2))\leq L_1\|\theta_1-\theta_2\|_2^{\alpha_0}$ 
	for some $\alpha_0>0$ and some $L_1>0$.
	Suppose $h(f(x|\theta_1),f(x|\theta_2))\leq L_2\|\theta_1-\theta_2\|_2^{\beta_0}$ 
	for some $\beta_0>0$ and some $L_2>0$. Here $\theta_1,\theta_2$ are any distinct elements in $\Theta_1$.
	\end{comment}
	\end{enumerate}

The maximum between two numbers is denoted by $a\vee b$ or $\max\{a,b\}$. 
An useful quantity is the average sequence length $\bar{\m}_{\n}=\frac{1}{\n}\sum_{i=1}^{\n}\m_i$. The posterior contraction theorem will be characterized in terms of distance $D_{\bar{\m}_{\n}}(\cdot, \cdot)$, which extends the original notion of distance $D_\m(\cdot,\cdot)$ by allowing real-valued weight $\bar{\m}_{\n}$. 

	\begin{thm}
    \label{thm:posconnotid} 
	Fix $G_0\in\Ec_{k_0}(\Theta^\circ)$. Suppose  \ref{item:prior}, 
		\ref{item:kernel} and additionally \eqref{eqn:genthmcon} hold. 
	  There exists some constant $C(G_0)>0$ such that as long as $n_0(G_0,\Gc_{k_0}(\Theta_1)) \vee n_1(G_0) \leq \min_{i} \m_i \leq \sup_{i} \m_i < \infty$, for every $\bar{M}_\n\to \infty$ there holds
	\begin{align*}
	&\Pi \biggr ( D_{\bar{\m}_{\n}}(G,G_0) \geq C(G_0)\bar{M}_\n \sqrt{ \frac{\ln(\n\bar{\m}_{\n})}{\n} }, \\
    & G\in \Ec_{k_0}(\Theta_1) \biggr |X_{[\m_1]}^1, \ldots,X_{[\m_{\n}]}^{\n} \biggr ) \to  0
	\end{align*}
	in $\bigotimes_{i=1}^{m}\P_{G_0,\m_i}$-probability as $\n\to \infty$.
	\end{thm}

\begin{rem} \label{rem:4.9}
By Theorem \ref{thm:idenmixpro}, $ n_0(G_0,\cup_{k\leq k_0}\Ec_k(\Theta))\leq 2k_0 -1$ for any $G_0\in \Ec_{k_0}(\Theta)$. While the uniform upper bound $2k_0-1$ might not be tight, it does show  $n_0(G_0,\cup_{k\leq k_0}\Ec_k(\Theta_1))<\infty$. 
$n_1(G_0) < \infty$ is a direct consequence of \eqref{eqn:genthmcon}. Hence $n_0(G_0,\cup_{k\leq k_0}\Ec_k(\Theta_1)) \vee n_1(G_0) <\infty$.
\myeoe    
\end{rem}

	
    Roughly speaking, Theorem \ref{thm:posconnotid} produces the following posterior contraction rates. The rate toward mixing probabilities $p_i^0$ is $O_P((\ln(\n\bar{\m}_\n)/\n)^{1/2})$. 
   Individual atoms $\theta_i^0$ have much faster contraction rate, which utilizes the full volume of the data set: 
   \begin{equation}
   O_P\left(\sqrt{\frac{\ln(\n\bar{\m}_\n)}{\n\bar{\m}_\n}}\right) = 
   O_P\biggr (\sqrt{\frac{\ln (\sum_{i=1}^{\n} \m_i)}{\sum_{i=1}^{\n} \m_i}} \biggr). \label{eqn:atomconvergencerate}
   \end{equation}
    The theorem establishes that the larger the average length of observed sequences, the faster the posterior contraction as $m\to \infty$. 

 Finally, the conditions \ref{item:kernel} and \eqref{eqn:genthmcon} can be verified for full rank exponential families and hence we have the following corollary from Theorem \ref{thm:posconnotid}. 

\begin{cor}  \label{cor:posconexp} 
Consider a full rank exponential family for kernel $P_\theta$ specified as in Theorem \ref{cor:expfamtheta} and assume all the requirements there are met. Fix $G_0\in \Ec_{k_0}(\Theta^\circ)$. Suppose that $\ref{item:prior}$ holds with $\Theta_1\subset \Theta^\circ$. Then  the posterior contraction rate of Theorem \ref{thm:posconnotid} holds.
\end{cor}

As already discussed, inverse bounds \eqref{eqn:genthmcon} and \eqref{eqn:curvatureprodbound} may be established for fairly general kernels that beyond exponential family, so do the posterior contraction rates; see \cite[Sections~5.3 and 6]{wei2022convergence} for more details.

\subsection{Hierarchical Dirichlet processes and beyond}
\label{sec:hierarchical}

In this section, we examine the roles of optimal transport in a full-fledged setting of hierarchical modeling, by illustrating with the canonical formulation of hierarchical models first described in Eq. \eqref{eqn:hierarchical} (see Section \ref{sec:deFinetti} for the motivation). This is a latent structured model that arises in the analysis of multiple samples of data populations: we are given $m$ sequences, each of which is made of $N$ exchangeable observations.

According to the model specification \eqref{eqn:hierarchical}, the latent structure that is deployed to analyze the $m$ data populations consists of $m$ corresponding mixing measures $G_1,\ldots, G_m$, as well as de Finetti's mixing measure $\Dscr$ that governs the generating process for the $G_i$'s. In a multiple sample setting, the objectives of inference may be multi-pronged: for instance, one might be interested in the inference of only one or more specific data populations among them, or a comparative analysis among such data populations. In this case the de Finetti's measure $\Dscr$ may be viewed merely as a modeling device to quantify a Bayesian uncertainty about the generating mechanism for a specific $G_i$ of interest; such uncertainty is expected to be reduced once the sample size $N$ associated with the  $i$-th sample increases, even as $m$ stays fixed \citep{CatalanoEtAlPosteriorAsymptoticsBoosted2022}. Consistent with the focus of this article, we are motivated by the second ramification of de Finetti's theorem, as discussed in Section \ref{sec:deFinetti} and investigated in \cite{nguyen2016borrowing}, particularly in the learning of the entire latent structure containing $\Dscr$ and the latent instances of $\{G_1,\ldots,G_m\}$. Similar to the results obtained for the mixtures of product distributions described in Section \ref{sec:mixproduct}, as the number of sampled sequences $m$ increases, one expects that $\Dscr$ can be learned more accurately. In addition, it is interesting to study the impact of the sequence length $N$ on the quality of the estimates of the entire latent structure. 

The most common approach to estimation for the hierarchical model \eqref{eqn:hierarchical} in the existing literature is done in a Bayesian way. 
The latent structure $\Dscr$ is endowed with a prior distribution: $\Dscr \sim \Pi$. It is important at this point to specify what $\Pi$ is. Recall that the mixing measures $G_1,\ldots, G_m \in \Pcal(\Theta)$, the space of Borel probability measures on $\Theta$. Thus, $\Dscr$ must be an element of the space of probability measures on $\Pcal(\Theta)$, namely $\Dscr \in \Pcal(\Pcal(\Theta))$. It follows that $\Pi \in \Pcal(\Pcal(\Pcal(\Theta)))$. 

In the last several sections, in order to arrive at a more precise asymptotic theory we restricted the space of $G_i$'s to discrete measures with a bounded number of supporting atoms, in which case the prior $\Pi$ can be easily described by a finite number of parameters. However, here we shall adopt the richer and more powerful modeling approach of Bayesian nonparametrics, in which no such restriction is required. A well-known modeling choice for $\Dscr$ is to set it to be a Dirichlet distribution on the space of $\Pcal(\Theta)$: $\Dscr := \textrm{DP}(\alpha,G_0)$, for some parameters $\alpha>0$ and $G_0 \in \Pcal(\Theta)$ \citep{ferguson1973bayesian}. To complete the specification of $\Pi$, in keeping with the hierarchical Bayesian modeling spirit, the unknown $G_0$ is endowed with another Dirichlet prior: $G_0 \sim \textrm{DP}(\gamma,H)$. The induced prior $\Pi$ is known as the hierarchical Dirichlet process \citep{Teh-etal-06}. Other priors have been proposed, including the nested Dirichlet process prior \cite{Rodriguez-etal-08}, as well as hierarchical constructions using other types of normalized completely random measures \citep{camerlenghi2019distribution}.
Once the prior $\Pi$ is properly set, the posterior distribution of quantities of interest, namely $\Pi(\Dscr| \{X_{[N]}^{i}\}_{i=1}^{m})$ and $\Pi(G_1, \ldots | \{X_{[N]}^{i}\}_{i=1}^{m})$, may be obtained; more precisely it is their approximate samples that can be obtained via Markov Chain Monte Carlo techniques. 

To identify and quantify the convergence behavior of latent structure $\Dscr$, one needs a suitable metric. Similar to the consideration which led to the definition of composite distance in Eq. \eqref{eqn:composite}, optimal transport provides a natural distance metric for the (Bayesian) hierarchy of distributions which may extend to an arbitrary number of hierarchical levels \citep{nguyen2016borrowing}. Allow $\Theta$ to be a complete separable metric space (i.e., $\Theta$ is a Polish space), then by Prokhorov's theorem $\Pcal(\Theta)$ is also a Polish space with respect to the topology of weak convergence. Moreover, fix $r\geq 1$, and let $\Pcal_r(\Theta)$ be restricted to Borel probability measures on $\Theta$ that has bounded $r$-th moment (when $\Theta$ is bounded, then $\Pcal_r(\Theta)= \Pcal(\Theta)$). $\Pcal_r(\Theta)$ is metrized by $W_r$ as given in Eq. \eqref{eq:optdist} and reproduced here: for $G,G' \in \Pcal_r(\Theta)$
\[W_r^r(G,G')= \inf_{\kappa \in \TPlan(G,G')} \int \|\theta-\theta'\|^r \textrm{d} \kappa(\theta,\theta').\]
Again by Prokhorov's theorem, $\Pcal_r(\Pcal_r(\Theta))$ is a Polish space, and moreover it can be metrized  by the following optimal transport distance which
is induced by the metric $\Wr$ on $\Pcal_r(\Theta)$: for
$\Dscr,\Dscr' \in \Pcal_r(\Pcal_r(\Theta))$

\begin{equation}
\label{Eqn-W2-def}
\Wcalr(\Dscr,\Dscr') := \inf_{\Kcal \in \TPlan(\Dscr,\Dscr')} 
\biggr [ \int \Wr^r(G,G') \; \Kcal(\dd G,\dd G') \biggr ]^{1/r},
\end{equation}
where $\Kcal(\dd G,\dd G')$ denotes a coupling of $\Dscr$ and $\Dscr'$, namely, a joint distribution on $\Pcal_r(\Theta)\times \Pcal_r(\Theta)$ which projects to the marginal $\Dscr$ and $\Dscr'$, and $\TPlan(\Dscr,\Dscr')$ stands for the space of all such couplings. Note that the symbol $\Wr$ was reused, first for $W_r(G,G')$ and then $W_r(\Dscr,\Dscr')$ but we should not be confused as the context is clear from the arguments. Note also that due to the tightness of $\TPlan(\Dscr,\Dscr')$ (cf. Theorem 4.1 in \cite{Villani-08}), and the continuity of the cost function $W_r^r(G,G')$, the optimal coupling which defines the infimum problem in the above formulation exists.

To gain intuition about the optimal transport distances which have been defined on probability measures on $\Pcal(\Theta)$ and $\Pcal(\Pcal(\Theta))$, we shall present a useful result relating $\Wr(\Dscr, \Dscr')$ and the $W_r$ distance of the corresponding mean probability measures. We say $G$ to be the mean probability measure of $\Dscr$ and write $\int P \dd \Dscr = G$, if $\int P(A) \Dscr (\dd P) = G(A)$ for all measurable subset $A \subset \Theta$. As an example, if $\Dscr = \textrm{DP}(\alpha,G)$, a Dirichlet distribution on $\Pcal(\Theta)$ parametrized by the base probability measure $G$ and concentration parameter $\alpha>0$, then $\int P \dd \Dscr  = G$ indeed. 
The following is from \cite{nguyen2016borrowing} (Lemma 3.1):
\begin{lem}
\label{lem:Wmeanmeasure}
(a) if $\int P \dd \Dscr = G$ and $\int P \dd \Dscr' = G'$, and $W_r(\Dscr,\Dscr')$ is finite for some $r\geq 1$, then $W_r(G,G') \leq W_r(\Dscr,\Dscr')$.\\
(b) If $\Dscr=\textrm{DP}(\alpha, G)$ and $\Dscr' = \textrm{DP}(\alpha, G')$, for some $\alpha>0$. Then $W_r(\Dscr, \Dscr') = W_r(G,G')$, if both quantities are finite.
\end{lem}
The identity in part (b) of the above lemma can be extended without difficulty to hold generally for the cases that the random measures under $\Dscr$ and $\Dscr'$ share the same distribution for jumps, which are independent of the measures' atoms (a fact also noted in \cite{hierarchical2024}). This identity plays a crucial role in the analysis of posterior contraction for the latent structure $\Dscr$ that arises in the hierarchical Dirichlet process model (HDP). Specifically, let $P_{\Dscr,N}$ be the induced probability distribution for the $N$-sample $X_{[N]}$ for given by the hierarchical model \eqref{eqn:hierarchical}. That is,

\begin{align}
\label{Eqn-pXG}
P_{\Dscr,N}(\dd x_1,\ldots, \dd x_N) = \nonumber \\
\int \prod_{j=1}^{N} 
 \biggr ( \int (f(\dd x_j |\theta) G(\dd \theta) \biggr ) \Dscr (\dd G).
\end{align}

Under a suitable regularity condition on the kernel $f$, the following upper bound on the KL divergence of the $N$-sample's generating distributions can be obtained
\citep{nguyen2016borrowing} (Lemma 3.2):
\[K(P_{\Dscr,N},P_{\Dscr',N}) \lesssim N W_r^r(\Dscr,\Dscr').\]
Now, for the HDP model, if $\Dscr = \textrm{DP}(\alpha,G)$ and $\Dscr'=\textrm{DP}(\alpha, G')$ for any pair $G,G' \in \Pcal(\Theta)$, the identity of Lemma \ref{lem:Wmeanmeasure} (b) yields
\[K(P_{\Dscr,N},P_{\Dscr',N}) \lesssim N W_r^r(G,G').\]

As is expected by the reader at this point, the most demanding part of the theory lies in establishing a suitable inverse bound, which is in this case an upper bound of the Wasserstein distance $\Wr(G,G')$ in terms
of the \myredvsec{total} variation distance $V(P_{\Dcal,N}, P_{\Dcal',N})$, which is of the form
\begin{equation}
\label{Eqn-overview-1}
\Wr^r(G,G') \lesssim V(P_{\Dcal,N}, P_{\Dcal',N}) + A_N(G,G'),
\end{equation}
where $A_N(G,G')$ is a quantity that tends to 0 as $N\rightarrow \infty$.
The rate at which $A_N(G,G')$ tends zero depends only on additional structures on the base measure $G$. As a consequence of this inverse bound, one can establish that the posterior distribution of the latent structure $G$ (and hence $\Dscr= \textrm{DP}(\alpha, G)$), contracts to the truth, say, $G_0$ (and respectively, $\Dscr_0 = \textrm{DP}(\alpha, G_0))$ as the number of sampled sequences $m$, as well as the sequence length $N$ increase. For full details, see \cite{nguyen2016borrowing}.

\section{Concluding remarks}

This article is an exposition on some recent theoretical advances in learning latent structured models, with a primary focus on the fundamental and natural roles that optimal transport distances play in such a statistical theory. We aimed for what is in our view the most critical and novel ingredient in this theory: the motivation, formulation, derivation and ramification of inverse bounds, a rich collection of structural inequalities for latent structure models which connect the space of distributions of unobserved structures of interest to the space of distributions for observed data. We illustrate this theory by focusing on the classical mixture models, as well as the more modern hierarchical models that have been developed in Bayesian statistics, machine learning and related fields. Such selected domains of illustration were chosen due in large part to our bias and interest; however, they already cover a vast territory of applications and for which there have been a sufficiently rich collection of developments that have emerged in the past couple of decades. 

There are several open areas and recent developments that we wish to highlight. 
First, it would be of interest to extend the strong identifiability theory such as those described in Section \ref{sec:finite} and \ref{sec:identifiability} to the types of mixture models for more complex data types, such as networks, graphs and more complex forms of probability kernels \citep{aragam2023uniform,chakraborty2024learning,chakraborty_thesis,zhang2025bayesian}. We believe that the methodology described in Section \ref{sec:MMD1} provides a promising general direction.
Second, the minimax theory for weakly identifiable settings, as discussed in Section \ref{sec:weakidentifiability}, remains open. It is also of interest to develop novel estimation that accounts for the singularity structures in a direct way. The approach of \cite{wang2024estimating} provides one possible direction. Alternatively, one may place the problem of estimating the model's latent structure (e.g., mixing measure) as part of a more ambitious inference task \cite{do2024dendrogram}, which results in a pointwise optimal estimation method for the mixing measure as a byproduct.
Third, as discussed in Section \ref{sec:hierarchical}, the theory for hierarchical models remains underdeveloped, especially in attaining a more precise understanding of the roles of data set's dimensions on the efficiency of estimating de Finetti's measures of interest, especially in a infinite-dimensional setting. A recent progress in this question occurs in the analysis for Latent Dirichlet allocation, which utilizes a precise correspondence between that class of hierarchical models with a class of mixture product distributions, so that the advances on the latter class (such as those described in Section \ref{sec:mixproduct}) may be brought to bear on the former \citep{do2025dirichlet}. 
Last, but not least, although we discussed how optimal transport based distances may be useful for deriving estimation procedures with favorable theoretical properties, we did not consider computational aspects of such procedures. This is an important, interesting and fast-developing research area. Examples include the adoption of hierarchical optimal transport distances described in Section \ref{sec:hierarchical}, which helps to overcome the computational burden of posterior inference with hierarchical models \citep{ho2017multilevel}, the computational properties of a variety of related notions of optimal transport \citep{hierarchical2024}, as well as the utilization of slice-Wasserstein distances in \cite{doss2020optimal} in learning finite mixtures of Gaussians. It is likely that these computational ideas can be developed into richer classes of models such as the ones discussed in this article.

\vspace{.1in}
\noindent \textbf{Acknowledgement.} The initial draft of this article arises from a short course on "Optimal transport based theory for latent structured models" given by Long Nguyen at Bocconi University in Italy in Spring 2023. We thank Antonio Lijoi and Igor Pruenster for many valuable discussions on this topic, and for the hospitality given by them, other faculty members and students at Bocconi.

\bibliographystyle{abbrv}
\bibliography{library}


\end{document}